\documentclass{amsart}
\usepackage[utf8]{inputenc}
\usepackage{verbatim}
\usepackage{amsfonts}
\usepackage{mathtools}
\usepackage[english]{babel}
\usepackage{xcolor}
\usepackage[T1]{fontenc}
\usepackage{amsmath}
\usepackage{amsthm}
\usepackage{float}
\usepackage{wrapfig}
\usepackage{tabu}
\usepackage{MnSymbol}
\usepackage{tcolorbox}
\usepackage{tikz}
\usepackage{young}
\usepackage[all]{xy}
\usepackage{pifont}

\usepackage{stackrel}
\newcommand{\acts}{\curvearrowright}
\usepackage{mathtools}
\usepackage{mathrsfs}

\usetikzlibrary{arrows}
\usepackage[a4paper]{geometry}

\usepackage{amsthm}
\usepackage{amsfonts}
\usepackage{graphicx}
\usepackage[colorinlistoftodos]{todonotes}
\usepackage[colorlinks=true, allcolors=blue]{hyperref}

\newtheorem{theorem}{Theorem}[section]
\newtheorem{corollary}[theorem]{Corollary}
\newtheorem{lemma}[theorem]{Lemma}
\newtheorem{rmk}[theorem]{Remark}
\newtheorem{example}[theorem]{Example}

\newtheorem{definition}[theorem]{Definition}

\newcommand{\Set}[1]{\{ #1 \}}

\usepackage{dirtytalk}

\newcommand{\G}{\mathcal{G}}

\newcommand{\B}[1]{\mathcal{B}_{#1}}

\title{Rigidity, Generators and Homology of Interval Exchange Groups}
\date{\today}
\newcommand{\Gn}[1]{\mathcal{G}^{(#1)}}
\newcounter{theoremintro}
\newtheorem{thmintro}[theoremintro]{Theorem}
\newtheorem{rmkintro}[theoremintro]{Remark}

\newtheorem{qnintro}[theoremintro]{Question}

\author{Owen Tanner}
\address[Owen Tanner]{School of Mathematics and Statistics,
University of Glasgow,
\linebreak University Place, Glasgow G12 8QQ, United Kingdom}
\email[]{o.tanner.1@research.gla.ac.uk}
\urladdr{https://owentanner1997.wordpress.com/}

\begin{document}

\maketitle
\section*{Abstract}
Let $\Gamma$ be a dense countable subgroup of $\mathbb{R}$. Then, consider $IE(\Gamma)$; the group of piecewise linear bijections of $[0,1]$ with finitely many angles, all in $\Gamma$. We introduce and systematically study a family of partial transformation groupoids coming from inverse semigroups,  $\G_\Gamma$, that realise $IE(\Gamma)$ as a topological full group. This new perspective on the groupoid models $\G_\Gamma$ of $IE(\Gamma)$ allows us to better understand the underlying C*-algebras and to compute homology. We show that $H_*(\G_\Gamma)=H_{*+1}(\Gamma)$. We show $C^*_r(\G_\Gamma)$ is classifiable in the sense of the Elliott classification program of C$^*$-algebras. We then classify these groups via the Elliott invariant, showing $IE(\Gamma) \cong IE(\Gamma') \Leftrightarrow \Gamma=\Gamma'$ as subsets of $\mathbb{R}$. We relate the Elliott invariant to Groupoid Homology via Matui's HK Conjecture. We relate the homology of $IE(\Gamma)$ to the homology of $\Gamma$ using the recent framework developed by Li. We investigate in greater detail three key cases, namely if $\Gamma \subset \mathbb{Q}$, if $\Gamma \cong \mathbb{Z}^n$, and if $\Gamma$ is a ring. For these three cases, we study homology in greater detail and find explicit generating sets. 
\section{Introduction}
The story of topological full groups begins with Giordano-Putnam-Skau\cite{giordano1999full}, who introduced the topological full group of a Cantor minimal system as the group of homeomorphisms of the Cantor set that locally are powers of a homeomorphism $T$. Later, Matui (in \cite{matui2012homology}) defined the topological full group $\mathsf{F}(\G)$ of an ample groupoid $\G$. The idea is to study the unitary subgroup of the inverse monoid of open compact bisections, in effect piecing together partial symmetries into global symmetries. \ \\ \ \\
Since then, topological full groups have solved many existence questions for infinite simple groups with various finiteness properties. For example, they provided the first examples of:
\begin{itemize}
    \item Infinite simple, finitely generated amenable groups \cite{juschenkomonod}.
    \item Simple finitely generated groups of intermediate growth \cite{nekrashevych2018palindromic}.
    \item Simple groups separated by finiteness properties \cite{skipper2019simple}.
\end{itemize}
The philosophy of the program of topological full groups is that we would be able to determine information about $\mathsf{F}(\G)$ by studying the underlying groupoids $\G$. Many things are known now in this direction \cite{matui2014topological} \cite{li2022} \cite{nekrashevych2019simple} but some questions, especially determining amenability or the existence of free subgroups, remain mysterious \cite{extensiveamenability}. \ \\ \ \\
A particularly interesting class of topological full groups, which have attracted much attention recently is certain groups of interval exchanges \cite{matui2006some} \cite{chornyi2020topological} \cite{extensiveamenability} \cite{bon2018rigidity}. 

\begin{definition} [$IE(\Gamma)$]
Let $\Gamma$ be a countable dense additive subgroup of $\mathbb{R}$, containing $1$. Then, let $IE(\Gamma)$ denote the group of piecewise linear bijections $f$ of $[0,1]$ with finitely many angles, all in $\Gamma$. That is $\{ft-t \; : t \in [0,1] \} \subset_{fin} \Gamma$. 
\end{definition}
One reason for the interest in these groups is the connection to classical dynamics, where dynamical systems coming from interval exchanges have been popular to study for some time \cite{katokstepin}. For further information about the dynamical perspective on interval exchanges, we recommend the survey \cite{viana} and book \cite{katokhasselblatt}. \ \\ \ \\ Also, these groups have been studied from the perspective of geometric group theory. The reason for this is that due to results by Juschenko-Monod \cite{juschenkomonod} and Matui \cite{matui2006some}, whenever $\Gamma=\mathbb{Z} \oplus \lambda \mathbb{Z}$ (for some irrational $\lambda$) the derived subgroup of $IE(\Gamma)$ is a rare example of a simple, finitely generated amenable group. The existence of such groups was first shown in \cite{juschenkomonod}.
\ \\ \ \\
Another reason for the interest in $IE(\Gamma)$ when $\Gamma$ is finitely generated is to understand the group of interval exchange transformations.  
\begin{definition}[IET]
  Let IET be the group of right continious permutations $g$ of $[0,1]$ such that the set $\{gt-t, \; t \in [0,1] \}$ is finite. 
\end{definition}
\begin{rmkintro}[IET Locally Embeds into $IE(\Gamma)$, where $\Gamma \cong \mathbb{Z}^n$]
For any finite set of elements $\gamma_1,\gamma_2,...,\gamma_n \in \text{IET}$. Then $S=\{\gamma_it-t \; : t \in [0,1], i=1,...,n \}$ is finite.  Note $\mathbb{R}$ is locally polycyclic. Therefore, there exists some polycyclic subgroup of $\mathbb{R}$, $\Gamma (\cong \mathbb{Z}^n)$ such that $S \subset_{fin} \Gamma$. Then, for all $i$, $\gamma_i \in IE(\Gamma)$. 
\label{local embedding}
\end{rmkintro}
There are outstanding open questions about IET. See for example \cite{de2013groupes}, \cite{extensiveamenability}. The first question is attributed to Katok. 
\begin{qnintro}
Does IET contain any nonamenable free groups? Is IET amenable? \label{question intro}
\end{qnintro}
Through Remark \ref{local embedding}, one approach to the above question is to study the same question about $IE(\Gamma)$ where $\Gamma$ is polycyclic. Note that the main result of \cite{juschenkomonod} establishes that $IE(\Gamma)$ is amenable whenever $\Gamma \cong \mathbb{Z}^2$ and the main result of \cite{extensiveamenability} establishes that $IE(\Gamma)$ is amenable whenever $\Gamma \cong \mathbb{Z}^3$. \ \\ 

We take a different perspective from the above papers. Instead of taking an action of $\Gamma/\mathbb{Z}$ on the Cantor space, we define a partial transformation groupoid $\alpha:\Gamma \acts X$ based on groupoids considered in \cite{xinlambda} which realise $IE(\Gamma)$ as a topological full group.
\begin{thmintro}
 $IE(\Gamma)$ is the topological full group of a minimal, free partial action of $\Gamma$ on the Cantor space. 
\end{thmintro}
Let $\G_\Gamma$ be the associated partial transformation groupoids. We give a systematic study of the family of groupoids $\G_\Gamma$, their reduced C*-algebras $C^*_r(\G_\Gamma)$, the groupoid homology $H_*(\G_\Gamma)$, and their topological full groups $\mathsf{F}(\G_\Gamma)=IE(\Gamma)$.
\begin{thmintro}[Lemma \ref{its classifiable}]
  $C^*_r(\G_\Gamma)$ is classifiable in the sense of the Elliott classification program. 
\end{thmintro}
This answers a question posed in \cite{xinlambda}, Section 6, where Li asked if $C^*_r(\G_\Gamma)$ was $\mathcal{Z}$-stable. 
The Elliott invariant is computed (see Lemma \ref{elliot invariant}).  Through Elliott classification, this identifies $C_r^*(\G_\Gamma)$ with concrete C*-algebras in certain cases (see Corollaries \ref{id with uhf gpd}, \ref{id with putnams cstar}. The Elliott invariant recovers $\Gamma$ as a subset of $\mathbb{R}$, so we obtain the following classification result for the groups $IE(\Gamma)$:
\begin{thmintro}[Classification of $IE(\Gamma)$]
(Theorem \ref{classification of the groups}) Let $\Gamma,\Gamma'$ be dense additive subgroups of $\mathbb{R}$. Then, the following are equivalent:
\begin{itemize}
    \item $IE(\Gamma) \cong IE(\Gamma')$ as abstract groups
    \item $\Gamma=\Gamma'$ as subsets of $\mathbb{R}$
\end{itemize}
\end{thmintro} We remark that this is much stronger than saying $\Gamma \cong \Gamma'$ as abstract groups, for example, we see that $IE(2\pi \mathbb{Z} \oplus \mathbb{Z}) \not \cong IE(\pi \mathbb{Z} \oplus \mathbb{Z})$. Note that this classification in the case when $\Gamma$ is finitely generated can also be recovered as a Corollary of [\cite{bon2018rigidity}, Theorem 10.3]. \ \\ \ \\
We study the homology of $\G_\Gamma$. The groupoid homology of $\G_\Gamma$ is a shifted version of the group homology of $\Gamma$.
\begin{thmintro}(Lemma \ref{gpd hom})
  Let $\Gamma$ be a dense countable subgroup of $\mathbb{R}$, containing $1$. Then,  $H_*(\G_\Gamma)=H_{*+1}(\Gamma)$
\end{thmintro}
The key point of inspiration here is the computation of homology for groupoids in work by Li \cite{xinlambda}. 
In Theorem \ref{hk conjecture}, Matuis HK Conjecture is verified directly for $\G_\Gamma$ i.e. it is shown  there are isomorphisms: $$K_0(C^*_r(\G_\Gamma))\cong\bigoplus_{i=1}^\infty H_{2i-1}(\Gamma) \quad K_1(C^*_r(\G_\Gamma)) \cong \bigoplus_{i=1}^\infty H_{2i}(\Gamma)$$ 
We also use the framework of topological full groups to obtain homological information about $IE(\Gamma)$ in terms of the groupoid homology of $\G_\Gamma$. For example, 
Matuis AH conjecture was recently confirmed for a broad class of groupoids containing all of the groupoids $\G_\Gamma$ by Li in [\cite{li2022}, Corollary E]. This gives us the first homology group of $IE(\Gamma)$ in terms of the homology of $\Gamma$. 
\begin{thmintro}[AH Exact Sequence] \label{ah exact sequence}
    (Lemma \ref{ah conj}) There exists a long exact sequence:
    $$H_2(\mathsf{D}(IE(\Gamma)) \rightarrow H_3(\Gamma) \rightarrow \Gamma \otimes \mathbb{Z}_2 \rightarrow IE(\Gamma)_{ab}\rightarrow H_2(\Gamma) \rightarrow 0   $$ 
\end{thmintro}
From this concrete picture of the abelianisation we can say more about finite generatedness of $IE(\Gamma)$. 
\begin{thmintro}(Theorem \ref{ie finitely generated})
   Let $\Gamma$ be a dense countable subgroup of $\mathbb{R}$ containing 1. The following are equivalent:
   \begin{enumerate}

       \item $\Gamma$ is finitely generated.
       \item $D(IE(\Gamma))$ is finitely generated. 
       \item $IE(\Gamma)$ is finitely generated. 
   \end{enumerate}
\end{thmintro}
1. $\implies$ 2. follows by results of Matui and Nekrashevych, and was observed historically \cite{extensiveamenability} \cite{bon2018rigidity}. 2. $\implies$ 1. is shown to be general behavior (Corollary \ref{a finitely generated iff} formalises this). 2. $\implies$ 3. is a consequence of our results in homology, in particular, the long exact sequence seen in Theorem \ref{ah exact sequence}. 3. $\implies$ 1. is an elementary observation presumably known to experts. 
\ \\ \ \\
We also apply [\cite{li2022}, Corollary C] to describe the rational homology of $IE(\Gamma)$ and $ D(IE(\Gamma))$ in terms of the rational homology of $\Gamma$ (See Lemma \ref{rational homology}).
This summarises our general results, but we can be more precise for restricted example classes.  
\ \\ \ \\
We study the case when $\Gamma \subset \mathbb{Q}$ in Subsection \ref{subs rational}. The corresponding class of groupoids $\G_\Gamma$ are conjugate to the canonical AF groupoid models of UHF algebras in this case (Corollary \ref{id with uhf gpd}).  From this, we find an explicit infinite presentation of $IE(\Gamma)$ as the inductive limit of finite symmetric groups (Lemma \ref{ generating set when rational }). In this case, we show $IE(\Gamma)$ and $ D(IE(\Gamma))$ are rationally acyclic. We also compute the abelianisation to be $\mathbb{Z}_2$, finding an explicit short exact sequence $D(IE(\Gamma)) \hookrightarrow IE(\Gamma) \xrightarrow{sgn} \mathbb{Z}_2$ where $sgn$ is an analogue of the usual sign homomorphism $S_n \rightarrow \mathbb{Z}_2$. 
\ \\ \ \\
In Subsection \ref{subs gamma a ring}, we  study the case when $\Gamma=\mathbb{Z}[\lambda,\lambda^{-1}]$, taking the viewpoint that $IE(\Gamma)$ is the Lebesgue-measure-preserving subgroup of an irrational slope Thompson's group $V_\lambda$ on the interval, as studied in \cite{stein1992groups} \cite{burillo_nucinkis_reeves_2022} \cite{cleary2000regular}. In this case, we obtain a concrete generating set for $IE(\Gamma)$ (Lemma \ref{generating set ring}) and study the homology of $IE(\Gamma)$. 
\ \\ \ \\
The main result of Section \ref{finiteness properties section} is to find an explicit finite generating set of such $D(IE(\Gamma))$. 
\begin{thmintro}(Theorem \ref{generating set k} (See \cite{chornyi2020topological}, Proposition 8))
Let $\Gamma \cong \mathbb{Z}^{d+1}$ be a dense additive subgroup of $\mathbb{R}$ such that $1 \in \Gamma$. Then we have that $\Gamma/\mathbb{Z} \cong \mathbb{Z}^{d} \oplus \mathbb{Z}_k$. Let $k>9$ and $ d>1$. Then we describe a concrete generating set $S$ of $D(IE(\Gamma))$ such that $|S|=2d+4$. 
\end{thmintro} 
The proof of this Theorem is to describe an explicit subshift of $\{0,1\}^{\Gamma/\mathbb{Z}}$ realising $IE(\Gamma)$ as a topological full group and then apply the main result of \cite{chornyi2020topological}. The specific description of generators can be found in Theorem \ref{generating set k} and Theorem \ref{generating set theorem}. In particular, we have a concrete generating set $S$ with four elements for the simple group $D(IE( \mathbb{Z} \oplus \lambda_1 \mathbb{Z} \oplus \lambda_2 \mathbb{Z}))$ where $\lambda_1,\lambda_2$ are rationally independent  (see Example \ref{concrete rank 2 generating set}). This group is also known to be amenable by the main result of \cite{extensiveamenability}. We believe this is the first concrete finite generating set of a simple, finitely generated amenable group. 
\ \\ \ \\
In Subsection \ref{subs polycyclic}, we obtain homological information about $IE(\Gamma)$ for the case when $\Gamma \cong \mathbb{Z}^{d}$. We show $D(IE(\Gamma))$ is rationally acyclic iff $d=2$, and compute the abelianisation explicitly for the cases $d=2,3$: $$\Gamma \cong \mathbb{Z}^{d} \Rightarrow IE(\Gamma)_{ab}=\begin{cases} \mathbb{Z} \oplus \mathbb{Z}_2^2 & d=2 \ \\
\mathbb{Z} \oplus \mathbb{Z}_2^3 & d=3 \ \\
 \end{cases}$$
\ \\ 
\textbf{Acknowledgements:} The author would like to thank James Belk, Xin Li, and Alistair Miller for useful comments. The author has received funding from the European Research Council (ERC) under the European Union’s Horizon 2020 research and innovation programme (grant agreement No.
817597). This paper forms part of the PhD Thesis of the author. 
\section{Preliminaries}

\subsection{Groupoids and Partial Actions}
A \textit{groupoid} is a small category of isomorphisms. This is a set $\G$ with partially defined multiplication $\gamma_1 \gamma_2$ and everywhere defined involutive operation $\gamma \mapsto \gamma ^{-1}$, satisfying:
\begin{enumerate}
    \item Associativity: If $\gamma_1\gamma_2$ and $(\gamma_1\gamma_2)\gamma_3$ are defined, then $\gamma_2\gamma_3$ is defined and $(\gamma_1\gamma_2)\gamma_3=\gamma_1(\gamma_2\gamma_3)$
    \item Existence of $r,s$: The range and source maps $r(\gamma)=\gamma\gamma^{-1}$, $s(\gamma)=\gamma^{-1}\gamma$ are always well defined. If the product $\gamma_1\gamma_2$ is defined, then $\gamma_1=\gamma_1\gamma_2 \gamma_2^{-1}$, $\gamma_2=\gamma_1^{-1} \gamma_1 \gamma_2$. 
\end{enumerate}
A \textit{topological groupoid} is a groupoid endowed with a topology such that the multiplication and inverse operations are continious.

Elements of the form $\gamma\gamma^{-1}$ are called units and the space of units is denoted $\G^{(0)}$. The space of composable pairs is denoted by $\Gn{2}$. Given two subsets $B_1,B_2 \subset \G$, we define their product to be $B_1 B_2:=\Set{\gamma_1\gamma_2 \; : \; \gamma_i \in B_i, \; (\gamma_1,\gamma_2) \in \Gn{2}} $ \ \\
Units $u_1,u_2 \in \Gn{0}$ belong to the same $\G$\textit{-orbit} if there exists $\gamma \in \G$ such that $s(\gamma)=u_1, \; r(\gamma)=u_2$, and the orbit of $u$ is denoted $\G(u)$. If $\G(u)$ is dense in $\Gn{0}$ for all $u \in \Gn{0}$ we call the groupoid \textit{minimal}. \ \\
The \textit{isotropy group}  (denoted $\G_u$) of a unit $u \in \Gn{0}$ is the group $\Set{\gamma \in \G \; : \; s(\gamma)=r(\gamma)=u}$, if we have that every isotropy group is trivial, i.e. $\G_u=\Set{u}$ one says that the groupoid is \textit{principal}. A weaker condition is \textit{essentially principal} which means that such units are dense in the unit space i.e. $\overline{ \Set{u \in \Gn{0} \; : \; \G_u= \Set{u}}}=\Gn{0}$. This is related to effectiveness. $\G$ is said to be a \textit{effective} if for all $g \in \G \setminus \Gn{0}$ and all $B \in \mathcal{B}^k$ containing $g$, there exists some $g \in B$ such that $s(g) \neq r(g)$. In the case where $\G$ is second countable and Hausdorff, being effective is equivalent to essentially principal.  We now give one of our key examples for this paper- a \textit{transformation groupoid}. 
\ \\
Let $\Gamma \acts X$ be a (countable) group acting by homeomorphisms on a topological space $X$. Then we define the \textit{transformation groupoid} $\Gamma \ltimes X$ to be the set of pairs $(g,x) \in \Gamma \times X$ here composable pairs are of the form $(g,h(x))(h,x)$ and composition is given by $(g,h(x))(h,x)=(gh,x)$. Here $s(g,x)=(1,x), \; r(g,x)=(1,g(x))$ and the unit space is canonically identified with $X$. A basis of the topology is given by $(g.U)$ where $g \in \Gamma$. and $U$ is open in $X$. 

Another construction we study in this paper is a subtly different groupoid; the restriction of a transformation groupoid to a partial action. Here, let $Y \subset X$ be an open subset such that every orbit $\Gamma x \; x \in X$ has nontrivial intersection with $Y$.  Then let us define the subgroupoid
$$\Gamma \ltimes X |_Y^Y:=\{ (\gamma,x) \in \Gamma \ltimes Y \; : \; x, \gamma(x) \in Y \} $$
This forms a subgroupoid of the transformation groupoid. We call this a partial action $\alpha: \Gamma \acts Y$ on $Y$ and denote the associated groupoid $\Gamma \ltimes_\alpha Y$. 

We have many properties of the actions $\alpha$ translate into properties of the (partial) transformation groupoid:
\begin{itemize}
\item If the action is\textit{ free} (i.e. for all $g \in \Gamma, x \in X$ $gx=x \implies g=1$ ) iff the groupoid is principal. 
\item Similarly, if the action is essentially free (i.e. there exists a dense subset $Y \subset X$ such that $\Gamma \acts Y$ is a free action) iff the groupoid is essentially principal iff the groupoid is effective. 
    \item If the action is minimal (i.e. the orbit $Gx$ is dense in $X$ for all $x \in X$), iff the groupoid is minimal. 
    \item If the action is amenable (for example, if $\Gamma$ is amenable) then the groupoid is amenable. 
\end{itemize}
A crucial tool to understanding topological groupoids comes from studying their bisections. A bisection is a open subset $B \subset \G$ such that $ s: B \rightarrow s(B) \; r:B \rightarrow r(B)$ are homeomorphisms. We denote the space of compact open bisections by $\mathcal{B}^k$. A topological groupoid $\G$ is said to be \textit{ample} if $\mathcal{B}^k$ forms a basis of the topology on $\G$.

In the ample case, the set $\mathcal{B}^k$ is an inverse semigroup with respect to set multiplication, and pointwise inverses.
\begin{example}
    Let $\alpha: \Gamma \acts X$ be a partial action of a discrete group on a locally compact totally disconnected space. Then
    $\Gamma \ltimes X$ is ample. In this case, a basis for the open compact bisections is given by:
    $$(g,U) \; \; g \in \Gamma, \; \;  U,g(U) \subset X \text{ compact, open}$$
 
\end{example}
Also crucial to our perspective is the notion of a $G$-subshift. 
\begin{example}[Subshift over $G$]
 Let $G$ be a countable discrete group and $A$ be a finite alphabet. Let $Y \subset A^G$ be a closed subset of $A^G$ invariant under the shift by $G$ that is, if $y=\{y(g)\}_{g \in G} \in Y$, then for all $h \in G$, $hy=\{y(gh^{-1})\}_{g \in G} \in Y$. Then, consider the groupoid with elements:
 $$ (y,g,\hat{y}) \quad y,\hat{y} \in Y, \; g \in G, \; gy=\hat{y}$$
The unit space $ \G^{(0)}=\{(y,1,y) \; : \; y \in Y\}$
is identified with $Y$. 
Multiplication is given by: 
$$\times: \G^{(2)}=\{((y_1,g,y_2),(y_2,h,y_3)) \; : \; y_i \in Y, g,h \in G \}\rightarrow \G \quad  (y_1,g,y_2)\cdot(y_2,h,y_3)=(y_1,gh,y_2) $$
the inverse operation is given by
$ (y,g,\hat{y})^{-1}=(\hat{y}, g^{-1}, y)$. Finally, a basis of the compact open bisections is $(U,g,gU) $, where $U$ is a compact open subset of $Y$. 
The resulting groupoid is therefore ample. 

\end{example}
Finally, let us introduce the (unique by \cite{raad2023cdiagonals}) AF-groupoid models associated with UHF algebras. See \cite{matui2006some} for more discussion on AF groupoids. 
\begin{example}[UHF Groupoids]
Let $\{k_i\}_{i \in \mathbb{N}}$ be a sequence of natural numbers. Let $k(n)=\prod_{i=1}^n$. We associate with $\{k_i\}_{i \in \mathbb{N}}$ a Bratelli diagram:
$$ 1 \xrightarrow{k_1} k_1 \xrightarrow{k_2} k(2) \xrightarrow{k_3} k(3) \xrightarrow{k_4} ... \xrightarrow{k_n} k(n) \xrightarrow{k_{n+1}} ... $$
We associate a groupoid as follows. Let $\mathcal{R}_k$ be the full equivalence relation on $k$ points, the groupoid whose unit space is $k$ distinct points and arrow space is a unique arrow from the point $i$ to the point $j$. From a map $k \xrightarrow{k'/k} k'$ we can induce an inclusion map on groupoids given $\mathcal{R}_k \xrightarrow{k'/k} \mathcal{R}_{k'}$ by sending the arrow $i \mapsto j$ to the collection of arrows $i(k'/k)+m \mapsto j(k'/k)+m$ where $0 \leq m <k'/k$. We construct a groupoid $\G=\bigcup_{n \in \mathbb{N}} \mathcal{R}_{k(n)}$ as the inductive limit associated with the Bratelli diagram. This is known as a UHF groupoid and is the standard groupoid model of the UHF algebra associated with the supernatural number $\prod_{i=1}^\infty k_i$. It is ample Cantor, minimal, and principal. 
    \label{uhf groupoids}
\end{example}

\subsection{Topological Full Groups}
Using the concept of bisections, we are ready to define the topological full group of $\G$:
\begin{definition}[Topological Full Group (as in \cite{nekrashevych2019simple}, Definition 2.3)]
Let $\G$ be an \'etale Cantor groupoid. The topological full group, denoted $\mathsf{F}(\G)$ is the set of bisections $\gamma \in \mathcal{B}^k$ such that $s(\gamma)=r(\gamma)=\Gn{0}$. This is a group with respect to multiplication.  \label{groupoid tfg}
\end{definition}In other words, this is the unit group $U(\B{\G})$ of the inverse monoid of compact open bisections $\B{\G}$. Elements of topological full groups can also be thought of as homeomorphisms of the unit space:
\begin{lemma}Each element $\gamma \in \mathsf{F}(\G)$ defines a homeomorphism of the unit space given by:
$$f_\gamma=(r_{\restriction_B})\circ (s_{\restriction_B})^{-1}: (\G)^{(0)} \rightarrow (\G)^{(0)}$$
If $\G$ is effective, the map $f: \mathsf{F}(\G) \rightarrow Homeo(\G^{(0)}) \quad \gamma \mapsto f_\gamma$ is an injection. 
\label{TFG are homeo}\end{lemma}
This movement between perspectives is routinely used throughout this text.
\begin{lemma}
   Let $\Gamma \ltimes_\alpha X$ be the partial transformation groupoid of a discrete group acting on the Cantor space. Suppose that $\alpha$ is essentially free. Then the following groups are isomorphic:
   \begin{itemize}
       \item $\mathsf{F}(\Gamma \ltimes_\alpha X)$
       \item  The group of homeomorphisms $ \gamma \in Homeo(X)$ such that $ \forall x \in X, \exists g \in \Gamma, U \subset X$ compact, open neighbourhood of $x \text{ such that } \gamma|_U=f|_U$
       \item The group of homeomorphisms $\gamma \in Homeo(X) $ such that there exists a partition $X=\bigsqcup_{i=1}^n X_i$ into compact open subsets, and  $ g_1,...g_n \in \Gamma \text{ such that } \gamma|_{X_i}=\alpha(g_i)|_{X_i}$
   \end{itemize}\label{description of a tfg of a partial action}
\end{lemma}

Let us discuss the subgroup structure of topological full groups. 
\begin{definition}[$\mathsf{A}(\G)$]
Let $\G$ be an effective, Cantor groupoid with infinite orbits. Let us define an analogue of the infinite alternating group in $\mathsf{F}(\G)$, which we call the alternating group of $\G$. For open compact bisection $B_1,B_2$ with $s(B_1),r(B_1)=s(B_2),r(B_2)$ pairwise disjoint, let 
$$\gamma_{B_1,B_2}=B_1 \sqcup B_2 \sqcup (B_1 B_2)^{-1} \sqcup (\G^{(0)} \setminus s(B_1) \sqcup s(B_2) \sqcup r(B_2))$$
We define $\mathsf{A}(\G)=\langle \gamma_{B_1,B_2} \; : \; B_1,B_2 \in \mathcal{B}^k, \; s(B_1),r(B_1)=s(B_2),r(B_2) \text{ are pairwise disjoint } \rangle $\end{definition}
\begin{definition}[$\mathsf{D}(\G)$]
    Let $\G$ be an effective Cantor groupoid. Let $\mathsf{D}(\G)$ denote the derived subgroup of $\mathsf{F}(\G)$, that is $\mathsf{D}(\G)=\langle [\gamma_1,\gamma_2] \; : \; \gamma_1,\gamma_2 \in \mathsf{F}(\G) \rangle $
\end{definition}
Note in particular then $A(\G)$ is a subgroup of $\mathsf{D}(\G)$. Indeed if we take for $B \in \mathcal{B}^k$ with $s(B) \cap r(B) = \emptyset$ 
$$ \gamma_{B}:=B \sqcup B^{-1} \sqcup (\G^{0} \setminus s(B) \cup r(B) ) \in \mathsf{F}(\G)$$
Then $$ \gamma_{B_1,B_2}=[\gamma_{B_1},\gamma_{B_2}]$$ It is currently open whether there exists an effective, \'etale, Cantor groupoid $\G$ such that $\mathsf{A}(\G) \neq \mathsf{D}(\G)$. The above group elements $\gamma_B$ again generate a certain subgroup. 
\begin{definition}[$\mathsf{S}(\G)$]
    Let $\G$ be an effective Cantor groupoid with infinite orbits. Let us define an analogue of the infinite symmetric group in $\mathsf{F}(\G)$, which we call the symmetric group of $\G$. Then $$\mathsf{S}(\G)=\langle \gamma_B \; : \; B \in \mathcal{B}^k \; s(B) \cap r(B)=\emptyset \rangle $$ \label{def symmetric}
\end{definition}The above argument also shows that $\mathsf{A}(\G)$ is a subgroup of $D(\mathsf{S}(\G))$. \ \\ 
Let us remark that each of the groups we have defined thus far are actually normal in $\mathsf{F}(\G)$. The normality of $\mathsf{D}(\G)$ is generic, but for $\mathsf{A}(\G),\mathsf{S}(\G)$ it follows from the observation that for all $\gamma \in \mathsf{F}(\G)$, and $\gamma_B \in \mathsf{S}(\G)$, $\gamma \gamma_B \gamma^{-1}=\gamma_{\gamma B \gamma^{-1}}$. 
Simplicity of $\mathsf{A}(\G)$ is equivalent to minimality of $\G$.

\begin{theorem}[Matui (\cite{matui2014topological} Theorem 4.16, see also Nekrashevych \cite{nekrashevych2019simple}, Theorem 1.1)]
Let $\G$ be an almost finite, effective Cantor groupoid. Then, the following are equivalent:  \label{d simple}
\begin{itemize}
    \item $\G$ is minimal. 
    \item $\mathsf{D}(\G)$ is simple. 
\end{itemize}
Moreover, in the case where $\G$ is minimal, $\mathsf{D}(\G) \cong \mathsf{A}(\G)$. 
\end{theorem}
A crucial result in the theory of topological full groups is the Matui-Rubin isomorphism Theorem. 
\begin{theorem}[Matui-Rubin Isomorphism Theorem (\cite{matui2014topological}, Theorem 3.10)]
\label{matui isomorphism theorem}
Let $\G_1,\G_2$ be essentially principal, \'etale, minimal Cantor groupoids. Then the following are equivalent:
\begin{itemize}
    \item $\G_1 \cong \G_2$ as \'etale groupoids. 
    \item $\mathsf{F}(\G_1) \cong \mathsf{F}(\G_2)$ as discrete groups.
     \item $\mathsf{D}(\G_1) \cong \mathsf{D}(\G_2)$ as discrete groups.
\end{itemize}
\end{theorem}

Let us now apply what we have learned to the UHF case. 
\begin{example}[UHF Topological Full Groups]
 Following on from Example \ref{uhf groupoids}, let us remark firstly that it is clear that $\mathsf{F}(\mathcal{R}_n)=S_n$ the symmetric group on $n$ generators. Let $\{k_i\}_{i \in \mathbb{N}}$ be a sequence of natural numbers. Let $k(n)=\prod_{i=1}^n$. We associate with $\{k_i\}_{i \in \mathbb{N}}$ a Bratelli diagram:
$$ 1 \xrightarrow{k_1} k_1 \xrightarrow{k_2} k(2) \xrightarrow{k_3} k(3) \xrightarrow{k_4} ... \xrightarrow{k_n} k(n) \xrightarrow{k_{n+1}} ... $$ It is not hard to see from here that if we have $\G$ as the groupoid associated to the above Bratelli diagram that $\mathsf{F}(\G)=\bigcup_{n \in \mathbb{N}}S_{k(n)}$ where the inclusion maps the cycle $(i,j) \in S_{k(n)}$ to the product of cycles $\prod_{1 \leq m < k_{n+1}}(i k_{n+1}+m, j k_{n+1}+m) $ in $S_{k(n+1)}$. Likewise, $\mathsf{A}(\G)=\mathsf{D}(\G)=\bigcup_{n \in \mathbb{N}}A_{k(n)}$, an inductive limit of (simple) alternating groups. Since we are in the setting of Theorem \ref{d simple}, we can also see that $\mathsf{A}(\G)$ is simple through this result. 
\end{example}
However, we see in the example above that $\mathsf{A}(\G)$ cannot be finitely generated, as the inductive limit of finite groups. This is not always the case, in fact in many cases the alternating group can be a simple, finitely generated group. Nekrashevych showed that a topological full groups alternating group is finitely generated if the underlying groupoid has a technical condition known as expansivity. We recall the definition below:
\begin{definition}[Expansive \cite{nekrashevych2019simple}]
For  an \'etale Cantor groupoid $\G$:
\label{expansive groupoid}
\begin{itemize}
    \item A compact set $K \subset \G$ is called a compact generating set such that $\G= \bigcup_{n \in \mathbb{N}}(K \cup K^{-1})^n$.
    \item A finite cover $\mathcal{B}=\{B_{i}\}_{i=1}^N$ of bisections is called expansive if $\bigcup_{n \in \mathbb{N}} (\mathcal{B} \cup \mathcal{B}^{-1})^n$ forms a basis for the topology of $\G$. 
    \item $\G$ is called expansive if there is a compact generating subset $K$ with an expansive cover $\mathcal{B}$.
    \label{expansive nekrashevych}
\end{itemize}\end{definition}
In Nekrashevych's paper \cite{nekrashevych2019simple}, Nekrashevych showed that this notion of expansivity is related to the finite generation of $\mathsf{A}(\G)$.
\begin{theorem}Let $\G$ be an expansive Cantor groupoid with infinite orbits. Then $\mathsf{A}(\G)$ is finitely generated [\cite{nekrashevych2019simple}, Theorem 5.6]. \label{finite generatedness theorem}
\end{theorem}
    This notion of expansivity generalises the notion of expansivity for (finitely generated) group actions. 
\begin{definition}[Expansive action]
    Let $\alpha:G \acts X$ be an action. We say that $\alpha$ is expansive if there exists $\epsilon >0$ such that for all $x,y \in X, \, x \neq y$, there exists $g \in G$ such that $d(g x, g y)>\epsilon$. \label{expansive action}
\end{definition}
Nekrashevych showed that a transformation groupoid of a finitely generated group acting on the Cantor space was expansive iff the underlying action was expansive.
\begin{lemma}
    Let $G \ltimes X$ be a transformation groupoid of a discrete group on the Cantor space $X$. Then $G \ltimes X$ is compactly generated iff $G$ is finitely generated. \label{compactly generated iff}
\end{lemma}
\begin{proof}
$\Leftarrow$ If $G$ is finitely generated, then it has a generating set $g_1,..,g_n$. Consider the compact set $K=\sqcup_{i \in I} (g_i, X)$. This clearly will then generate the groupoid. \ \\
$\Rightarrow$ If $G \ltimes X$ is compactly generated, there exists a compact subset $K$ generating the groupoid. But this groupoid is ample; it is generated by compact open bisections. Therefore there exists a finite set of group elements $\{g_i\}_{i=1}^n$ subsets $\{Y_i\}_{i=1}^n$ such that $K=\sqcup_{i=1}^n (g_i,Y_i)$. Then, $K \subset K'=\sqcup_{i=1}^n(g_i,X)$, so $K'$ is a compact generating set. Hence, $\{g_i\}_{i=1}^n $ is a finite generating set of $G$. 
\end{proof}
\begin{lemma}
    Let $\G$ be an essentially principal ample groupoid with infinite orbits. Then $\mathsf{A}(\G)$ is finitely generated $\implies \G$ is compactly generated. \label{a finitely generated implies}
\end{lemma}
\begin{proof}
    Let $g \in \G$. It is enough to show that there exists some bisection $B \in \mathsf{A}(\G)$ such that $g \in B$, since then the generating set of $\mathsf{A}(\G)$ also serves as a compact generating set for $\G$. If $s(g) \neq r(g)$, by ampleness, we have there exists some compact open bisection $\hat{B}_1$ such that $g \in \hat{B}_1$. By restricting the source of $\hat{B}_1$ if necessary, we may assume that $s(\hat{B}_1),r(\hat{B_1})$ are disjoint with $s(\hat{B}_1) \bigcup r(\hat{B_1}) \neq \G^{(0)}$. By minimality, we have there exists some $h \in \G$ with $s(h)=r(g)$ and $r(h) \in (s(\hat{B}_1) \bigcup r(\hat{B_1}))^c$ Then by ampleness, there exists some compact open bisection $B_2$ with $h \in B_2$. Again, by restricting $s(B_2)$ if necessary, we can assume $s(B_2) \subset r(\hat{B}_1)$ and $r(B_2) \subset (s(B_2) \bigcup \hat{B}_1^{-1}(s(B_2))^c$. Now let $B_1=\hat{B}_1 |_{\hat{B_1}^{-1} s(B_2)}$. Then 
$$ \gamma_{B_1,B_2} =B_1 \sqcup B_2 \sqcup (B_1B_2)^{-1} \cup (\G^{(0)} \setminus s(B_1) \cup s(B_2) \cup r(B_2))\in \mathsf{A}(\G)$$
Satisfies $g \in \gamma_{B_1,B_2} $. \ \\
Otherwise $s(g)=r(g)$. Then, there exists $g_1,g_2$ such that $s(g_i) \neq r(g_i)$ but $g_1g_2=g$. Hence by the above argument, there exists $\gamma_1,\gamma_2 \in \mathsf{A}(\G)$ such that $g_1 \in \gamma_1, \gamma_2 \in B_2$ and hence $g=g_1g_2\in \gamma_1 \gamma_2 \in \mathsf{A}(\G)$. 
\end{proof}
Remark that the converse to Lemma \ref{a finitely generated implies} fails, an explicit counterexample would be any $\mathbb{Z}$-odometer.
This proof is inspired by that of [\cite{matui2014topological}, Lemma 3.7]. 
Combining the previous two Lemmas, we obtain the following Corollary.

\begin{corollary}

    Let $\alpha:G \acts X$ be an expansive, minimal, essentially free action of a discrete countable group on the Cantor set. Then
        $G \text{ is finitely generated} \iff \mathsf{A}(G \ltimes X) \text{ is finitely generated.}$
   \label{a finitely generated iff}
\end{corollary}
\begin{proof}
$\Leftarrow$ Is a combination of Lemma \ref{a finitely generated implies} and Lemma \ref{compactly generated iff}.  \ \\
$\Rightarrow$ Follows via Theorem \ref{finite generatedness theorem}. 
\end{proof}

\section{Interval Exchange Groups as Topological Full Groups}
Let us first adapt $\mathbb{R}$ so that we can allow discontinuities at some subset $\Gamma$, by including two points $\tau_+,\tau_-$, separated in the topology, at each point $\tau \in \Gamma$. This definition follows the notation of \cite{chornyi2020topological}. 
\begin{definition}
Let $\Gamma \subset \mathbb{R}$. Let 
$\mathbb{R}_\Gamma:=\{t, a_+,a_- \; : \; t \in \mathbb{R} \setminus \Gamma, a \in \Gamma \} $
With the canonical quotient map $q$ onto $\mathbb{R}$.
$q: \mathbb{R}_\Gamma \rightarrow \mathbb{R} \quad t \mapsto t, a_{\pm} \mapsto a$
Let us define a total order on $\mathbb{R}_\Gamma$ by 
$q(x)<q(y) \implies x < y \; \forall x,y \in \mathbb{R}_\Gamma, a_-<a_+ \forall \tau \in \Gamma$
And let us topologise $\mathbb{R}_\Gamma$ by the order topology, i.e. the topology generated by open intervals:
$$ (x,y)=\{z \in \mathbb{R}_\Gamma \; : \; x<z<y \}, \; x,y \in \mathbb{R}_\Gamma$$
Let $\Gamma_\pm:=\{x_+,x_- \; : x \in \Gamma \} \subset \mathbb{R}_\Gamma$
\end{definition}
This makes the topology identifiable with the disjoint union of countably many Cantor spaces. 
\begin{lemma}
    Let $\Gamma \subset \mathbb{R}$ be dense and countable. Then a (countable) basis for the topology on $\mathbb{R}_\Gamma$ is given by:
    $$(a_-,b_+) \; a< b, \; a,b \in \Gamma $$
    Moreover, each set of the form $(a_-,b_+)$ with $a<b$ is a Cantor set. 
\end{lemma}
\begin{proof}
First let us remark that for all $a<b$ that $(a_-,b_+)=[a_+,b_-]$, so that each of these sets are clopen. By density of $\Gamma$, these form a basis for the topology on $\mathbb{R}_\Gamma$. Thus we establish that $\mathbb{R}_\Gamma$ is second countable, with a basis of clopen sets. Note moreover that the basis elements clearly separate points in $\mathbb{R}_\Gamma$. Note that $q$ is continuous, indeed the preimage of $(a,b) \subset \mathbb{R}$ is of the form $(x,y)$ for some $x,y \in \mathbb{R}_\Gamma$. But since $q(a_-,b_+)=q([a_+,b_-])=[a,b]$, compactness follows. In total then we have that for all $a,b \in \Gamma$ $(a_-,b_+)$ is compact, and has a countable basis of compact open subsets; by Brouwer's theorem, we are done. 
\end{proof}

In \cite{xinlambda}, Section 2.3 Li constructs an analagous space as follows. Let $\Gamma$ be an additive subgroup of $\mathbb{R}$. Let $D(\Gamma^+)$ be the (abelian) semigroup C*-algebra of $\Gamma \cap [0,\infty)$. This group has the basis of idempotents $\{ 1_{a + \Gamma^+ } \; a \in \Gamma \cap [0,+\infty)\} $ He then considers the Gelfand dual space $\Omega(D(\Gamma^+))$, and removes the trivial character $\chi_\infty$ such that for all $a \in \Gamma$ $\chi_\infty(1_{a+\Gamma^+})=1$. This space is denoted $O_{\Gamma^+ \subseteq \Gamma}$ Concretely, this is the space of nonzero, nontrivial, characters
$ \chi: D(\Gamma^+) \rightarrow \{0,1\}$ that are strictly decreasing with on the basis (with respect to the partial order $1_{a + \Gamma^+} \leq 1_{b+ \Gamma^+ } \iff a \leq b$).  This space is topologised in the weak operator topology, which is the topology generated by the basic compact open sets $U_{a,b}:=\{\chi \in _{\Gamma^+ \subseteq \Gamma} \; : \; \chi(a)=1, \chi(b)=0 \}, a,b \in \Gamma$. There is a canonical homeomorphism:
$$ f: \mathbb{R}_\Gamma \rightarrow O_{\Gamma^+ \subseteq \Gamma} \quad a_+ \mapsto \chi_a^+, \; a_- \mapsto \chi_a^-, \; t \mapsto \chi_t \quad a \in \Gamma, t \in \mathbb{R}_\Gamma \setminus \Gamma_\pm  $$
Where for all $a,b \in \Gamma$, $t \in \mathbb{R}_\Gamma \setminus \Gamma_\pm  $,
$ \chi_a^+(1_{b+\Gamma^+})=1 \iff b\leq a, \;  \chi_a^-(1_{b+\Gamma^+})=1 \iff b<a, \; \chi_t(1_{b+\Gamma^+})=1 \iff b<t $. 
For all $a,b \in \Gamma, a<b$ $f[a_+,b_-])=U_{a,b}$. 
\ \\ \ \\
Let $\ell \in \Gamma \cap (0,+\infty)$. Then, let $q^*$ define a canonical inclusions of $[0,\ell]$ into $[0_+,\ell_-]$ that sends $(a,b] \mapsto (a_+,b_-]$ for all $a,b \in \Gamma$ with $a<b$:
$$q^*:[0,\ell] \hookrightarrow [0_+,\ell_-] \quad t \mapsto \begin{cases} t & t \nin \Gamma \ \\
0_+ & t=0 \ \\
t_- & t \in \Gamma \setminus \{0\}
\end{cases}$$
\begin{lemma}[$IE(\Gamma)$ as a topological full group of a partial action]
Let $\Gamma$ be a countable, dense additive subgroup of $\mathbb{R}$ with $1 \in \Gamma$. Let $\alpha: \Gamma \acts \mathbb{R}_\Gamma $
be the canonical additive action of $\Gamma$ on $\mathbb{R}_\Gamma$ (i.e. given by for each $c \in \Gamma_+$
$\alpha_c(a_\pm)=(a+c)_\pm, \; \alpha_c(t)=t+c \; \forall a,c \in \Gamma, \; \forall t \nin \Gamma  $) and consider the 
restriction of $\alpha$ to a partial action on $[0_+,1_-]$. Then we have that \label{IE as a partial action}
 $$\mathsf{F}(\Gamma \ltimes_{\alpha} [0_+,1_-]) \cong IE(\Gamma)$$
\end{lemma}
\begin{proof}A basis of the compact open bisections of the groupoid $\Gamma \ltimes_{\alpha} [0_+,1_-]$ is given by:
$$ (c,[a_+,b_-]) \; c \in \Gamma \quad a,b \in \Gamma, \quad \max\{-c,0\} \leq a < b \leq \min\{1-c,1\}  $$
Elements of $\mathsf{F}(\Gamma \ltimes_{\alpha} [0_+,1_-])$ are homeomorphisms $f$ of $[0_+,1_-]$ for which there exists a finite subset $\{x_i\}_{i=1}^n \subset \Gamma$, with $0=x_1<x_2<...<x_n=1$, and elements $\{c_i\}_{i=1}^n \subset \Gamma$.
$$f|_{[(x_i)_+,(x_{i+1})_-]}=\alpha(c_i)|_{[(x_i)_+,(x_{i+1})_-]}$$

Hence, making use of $q,q^*$ we can obtain the isomorphism:
$$\varphi: \mathsf{F}(\Gamma \ltimes_{\alpha} [0_+,1_-])  \rightarrow IE(\Gamma) \quad g \mapsto qgq^* $$
Under $\varphi$, it is clear that for any element $f \in \mathsf{F}(\Gamma \ltimes_{\alpha} [0_+,1_-]) $, $\varphi(f)$ is a right continuous piecewise linear bijection of $[0,1]$ with finitely many angles $\{c_i\}_{i=1}^n \subset \Gamma$. Note this is a group homomorphism since $qq^*=Id_{[0,1]}$. If $\varphi(f)=1$, then $c_i=0$ for all $i$, hence $f \in \mathsf{F}(\Gamma \ltimes_{\alpha} [0_+,1_-])$ is the identity, hence $\varphi$ is an isomorphism.
\end{proof}Let us also remark that for $\alpha$, we have that $[0_+,\ell_-]$ is $\Gamma \ltimes \mathbb{R}_\Gamma$ full (by density, for all $x \in \mathbb{R}_\Gamma$, we can choose $a \in \Gamma$ such that $0<q(x)-a<\ell$). 
\begin{rmk}
 The above partial action is conjugate (via the homeomorphism $f$) to the groupoid model $\Gamma \ltimes O_{\Gamma^+ \subseteq \Gamma}|_{N(\Gamma^+)}^{N(\Gamma^+)}$ of $\mathcal{F}^\lambda$ in \cite{xinlambda}, Section 2.3. 
\end{rmk}

It is also relatively straightforward to see an alternative, equivalent description of $IE(\Gamma)$ in terms of a group action. 
\begin{lemma}[$IE(\Gamma)$ as the topological full group of an action]
   Let $\Gamma$ be a countable, dense additive subgroup of $\mathbb{R}$ with $1 \in \Gamma$. Then 
   $$\hat{\alpha}: \Gamma/\mathbb{Z} \acts [0_+,1_-]$$
   Let $[t]$ denote the equivalence class of $t \in \mathbb{R}$ ($\text{mod } \mathbb{Z}$). 
   $$\hat{\alpha}([c])(t)=[t+c] \quad \hat{\alpha}([c])(a_{\pm})=[a+c]_\pm$$
   We have that $IE(\Gamma)=\mathsf{F}(\Gamma \ltimes_{\hat{\alpha}} [0_+1_-] )$
   \label{IE as a global action}
\end{lemma}
We omit this proof since this perspective agrees with the perspective as in \cite{bon2018rigidity} \cite{extensiveamenability} \cite{chornyi2020topological}, and the proof of this follows the same proof as that of Lemma \ref{IE as a partial action}.

It is time to remark on some basic facts for these partial action groupoids:
\begin{lemma}[Regularity of $\alpha$] Let $\Gamma$ be a dense countable subgroup of $\mathbb{R}$ with $1 \in \Gamma$. Then
$\alpha:\Gamma \acts \mathbb{R}_\Gamma$ is a free, amenable, minimal action. 
\end{lemma}
\begin{proof}
 The action is amenable since $\Gamma$ is an amenable group. Suppose for some $c \in \Gamma $, and some $x \in \mathbb{R}_\Gamma$, $\alpha(c)(x)=x$. Then in particular in $\mathbb{R}$, $q(\alpha(c)x)=q(x)-c=q(x) \Rightarrow c=0$. This establishes that $\alpha$ is free. \ \\
Now let us establish minimality. First let us remark the following convergence rules in $\mathbb{R}_\Gamma$:
$$\lim_{n \rightarrow \infty} x_n=x_+ \iff \lim_{n \rightarrow \infty} q(x_n)=q(x) \text{ from above }$$
$$\lim_{n \rightarrow \infty } x_n=x_- \iff lim_{n \rightarrow \infty} q(x_n)=q(x) \text{ from below }$$
$$ \lim_{n \rightarrow \infty } x_n=x \text{ s.t. } x \nin \Gamma_\pm \iff \lim_{n \rightarrow \infty } q(x_n)=x$$
For all $x \in \mathbb{R}_\Gamma$, the image of the orbit $\Gamma x$ under $q$, $q(\Gamma x)=q(x)+ \Gamma \subset \mathbb{R}$ is dense in $\mathbb{R}$, so we can find sequences in $q(x)+\Gamma$ tending to any $x' \in \Gamma$ from above or below, and sequences approximating any $x' \nin \Gamma$ in $q(\Gamma x)$. Minimality follows.  
\end{proof}
\begin{corollary} Let $\Gamma$ be a dense countable subgroup of $\mathbb{R}$ with $1 \in \Gamma$.
As a groupoid, $\Gamma \ltimes_\alpha [0_+,1_-]$ is minimal, principal, and amenable. \label{groupoid regularity}
\end{corollary}
Note in particular as a Corollary of Matui's isomorphism Theorem (Theorem \ref{matui isomorphism theorem}), and the regularity established for $\alpha$, $\hat{\alpha}$ the identification of the topological full groups implies the groupoids coming from $\alpha$ and $\hat{\alpha}$ are conjugate. 
\begin{corollary} Let $\Gamma$ be a dense countable subgroup of $\mathbb{R}$ with $1 \in \Gamma$.
    As groupoids, $\Gamma/\mathbb{Z} \ltimes_{\hat{\alpha}} [0_+,1_-] \cong \Gamma \ltimes_\alpha [0_+,1_-] $
\end{corollary}
This identification establishes that the groupoid is almost finite. 
 \begin{lemma}
 Let $\Gamma$ be a dense countable subgroup of $\mathbb{R}$ with $1 \in \Gamma$. Then the groupoid  $\Gamma \ltimes_\alpha [0_+,1_-]$ is an almost finite groupoid. 
 \end{lemma}
 \begin{proof}
Using that $\Gamma \ltimes_\alpha [0_+,1_-] \cong \Gamma / \mathbb{Z} \ltimes_{\hat{\alpha}} [0_+,1_-]$
can be written as the transformation groupoid of a free minimal action of $\Gamma/\mathbb{Z}$ (a countable amenable group) on the Cantor space. Moreover, since $\Gamma/\mathbb{Z}$ is in particular abelian, all finitely generated subgroups have polynomial growth. Free minimal actions of such groups were shown to be almost finite in [\cite{kerr2020almost}, Theorem C], and so the underlying groupoid is almost finite. 
 \end{proof}
 We then fit into the scope of Theorem \ref{d simple}, identifying $\mathsf{A}(\Gamma \ltimes_\alpha [0_+,1_-])$ with $\mathsf{D}(\Gamma \ltimes_\alpha [0_+,1_-])$ and establishing simplicity.
\begin{corollary}
Let $\Gamma$ be a dense countable subgroup of $\mathbb{R}$ with $1 \in \Gamma$. Then  $D(IE(\Gamma))=\mathsf{A}(\Gamma \ltimes_\alpha [0_+,1_-])$ is simple.  \label{derived ie is simple}
\end{corollary}
Another corollary we have as a consequence of almost finiteness is that the associated crossed product is $\mathcal{Z}$-stable. In fact, we now have enough to say that the associated crossed product is classifiable. When we say classifiable, we mean in the sense of the Elliott classification program see [\cite{tikuisis2017quasidiagonality}, Corollary D]. Here we substitute finite nuclear dimension for $\mathcal{Z}$-stability by using [\cite{castillejos2021nuclear}, Theorem A]
\begin{theorem}[Classification] \label{classification}
Let $A,B$ be simple, unital, separable, nuclear, infinite dimensional, $\mathcal{Z}$-stable C$^*$-algebras satisfying the UCT. Then, $$A \cong B\iff Ell(A) \cong Ell(B)$$
\end{theorem}

\begin{lemma}
 Let $\Gamma$ be a dense countable subgroup of $\mathbb{R}$ with $1 \in \Gamma$. Then the C$^*$ algebra given by
$C([0_+,1_-])\ltimes_\alpha \Gamma$ is classifiable.  
\label{its classifiable}
   
\end{lemma}
\begin{proof}
  $\mathcal{Z}$-stability is established by almost finiteness via [\cite{kerr2020dimension}, Theorem 12.4]. Let us note that the group $\Gamma/\mathbb{Z}$ is always abelian, so, in particular, has polynomial growth. By construction, this is the groupoid C*-algebra of an amenable groupoid hence by [\cite{tu1999conjecture}, Proposition 10.7]  they satisfy the UCT. The other classifiability conditions can be read off from Corollary \ref{groupoid regularity}; the groupoid is minimal and principal (implying simplicity), with a compact unit space (implying unital), second countable (implying separable), and amenable (implying nuclearity).
\end{proof}
It would therefore be convenient to compute the Elliott invariant in this case. 
 
\begin{lemma}[Li]
Let $\Gamma$ be a dense countable subgroup of $\mathbb{R}$ with $1 \in \Gamma$. The Elliott invariant for $C([0_+,1_-])\ltimes_\alpha \Gamma$ is as follows. 
$$(K_0(C([0_+,1_-])\ltimes_\alpha \Gamma),[1]_0, K_1(C([0_+,1_-])\ltimes_\alpha \Gamma  ) \cong (K_1(C_r^*(\Gamma), [U_1]_1, K_0(C_r^*(\Gamma)/\mathbb{Z}[1]_0)$$ 
Where $[U_{1}]_1$ here is a unique trace $\tau$ on $K_0$ satisfying $\tau(K_0(C([0_+,1_-])\ltimes_\alpha \Gamma )))=\Gamma$. 
\ \\
In particular, $C([0_+,1_-])\ltimes_\alpha \Gamma \cong C([0_+,1_-]) \ltimes_\alpha \Gamma' \implies \Gamma=\Gamma'$ (as subsets of $\mathbb{R}$). 
\label{elliot invariant}
\end{lemma}
\begin{proof}
This proof directly generalises via Corollary 3.3 and Proposition 3.6 of \cite{xinlambda}. 
\end{proof}
Since $Ell (C^*_r(\G_\Gamma))$ recovers $\Gamma$ as a subset of $\mathbb{R}$, each of the C*-algebras are pairwise nonisomorphic. In turn then, so are the groups $IE(\Gamma)$:
\begin{theorem}
Let $\Gamma$ be a dense countable subgroup of $\mathbb{R}$ with $1 \in \Gamma$. Then, the following are equivalent:
\begin{enumerate}
   
    \item $\Gamma =\Gamma'$ (as subsets of $\mathbb{R}$)
    \item $IE(\Gamma)\cong IE(\Gamma')$ (as groups)
    \item $\mathsf{D}(IE(\Gamma))\cong\mathsf{D}(IE(\Gamma')) $ (as groups)
    \item $[0_+,1_-] \ltimes_\alpha \Gamma \cong [0_+,1_-] \ltimes_\alpha \Gamma'$ (as groupoids)
    \item $C^*_r([0_+,1_-] \ltimes_\alpha \Gamma)\cong C^*_r([0_+,1_-] \ltimes_\alpha \Gamma')$ (as C$^*$ -algebras)
    \item $Ell(C^*_r([0_+,1_-] \ltimes_\alpha \Gamma)) \cong Ell([0_+,1_-] \ltimes_\alpha \Gamma')) $
\end{enumerate}
\label{classification of the groups}
\end{theorem}
\begin{proof}
The implications $1. \Rightarrow 2. \Rightarrow 3.$, $4. \Rightarrow 5. \Rightarrow 6.$ are straightforward. $6 \Rightarrow 1.$ follows from Lemma \ref{elliot invariant}, since part of the Elliott invariant is the unique trace that recovers $\Gamma$. $3. \Rightarrow 4.$ is the Matui-Rubin isomorphism Theorem (Theorem \ref{matui isomorphism theorem}). 
\end{proof}
\begin{rmk}
This classification result can already be seen for the case of $\Gamma$ finitely generated as a Corollary of work by Matte-Bon [\cite{bon2018rigidity}, Theorem 10.3].
\end{rmk}
Note that this also identifies, via Theorem \ref{classification} many of the associated crossed products with concrete C*-algebras.
The K-Theory and tracial data of UHF algebras is computed in [\cite{davidson1996c}, Chapter III]. Note also that it was recently shown that UHF algebras have unique AF Cartan subalgebras [\cite{raad2023cdiagonals}, Theorem D], and hence the identification here is actually an identification on the groupoid level, due to Renault's reconstruction theorem [\cite{raad2022generalization}, Theorem 1.1]. 
\begin{corollary} \label{id with uhf gpd}
    Let $\Gamma \subset \mathbb{Q}$. Let $\{k(n)\}_{n \in \mathbb{N}}$ be a strictly increasing sequence of natural numbers such that $k(n) | k(n+1) $ and $\{1/(k(n))\}$ is a generating set of $\Gamma$. Then, $\Gamma \ltimes [0_+,1_-]$ is conjugate to the UHF groupoid associated with the Bratelli diagram
    $$ k(1) \xrightarrow{k(2)/k(1)} k(2) \xrightarrow{k(3)/k(2)} k(3) \xrightarrow{k(4)/k(3)} \hdots \xrightarrow{k(n)/k(n-1)} k(n) \xrightarrow{k(n+1)/k(n)} \hdots $$
    As in Example \ref{uhf groupoids}. This limit is independent of the choice of generating set, therefore such inductive limits of symmetric/alternating groups are classified by their supernatural numbers $\prod_{i=1}^\infty k_i$. \label{identification with UHF}
\end{corollary}
In \cite{modindex}, Li constructs $F^\lambda$ as in \cite{xinlambda} and \cite{modindex} as particular examples of our groupoids in the case where $\Gamma=\mathbb{Z}[\lambda,\lambda^{-1}]$. 
 \begin{corollary}
     \label{id with putnams cstar}
     Let $\Gamma=\mathbb{Z}[\lambda,\lambda^{-1}]$, where $\lambda \in \mathbb{R}$. Then, $C([0_+,1_-]) \ltimes_{\alpha} \Gamma \cong F^\lambda$ as in \cite{xinlambda} and \cite{modindex}. 
 \end{corollary}
\section{Concrete Generating Sets In The Polycyclic Case}
\label{finiteness properties section}
When $\Gamma$ is finitely generated, so is $D(IE(\Gamma))$. The aim of this section is to find a concrete generating set for $D(IE(\Gamma))$ in the case when $\Gamma $ is polycyclic. Let us first show that indeed the derived subgroup is finitely generated. 
\begin{lemma} Let $\Gamma$ be a dense countable subgroup of $\mathbb{R}$ with $1 \in \Gamma$. The action 
$\hat{\alpha}:\Gamma/\mathbb{Z} \acts [0_+,1_-]$ as in Theorem \ref{IE as a global action} is an expansive action in the sense of Definition \ref{expansive action}. 
\label{expansive action IE}
\end{lemma}

\begin{proof} Let $\lambda \in \Gamma \cap (0,1) $. Let $x,x' \in [0_+,1_-]$ be distinct. Let $\epsilon =d([0_+,(\lambda)_-],[(\lambda)_- , 1_+ ])$ We separate into two cases:
\begin{itemize}
    \item If $q(x) \neq q(x')$, suppose wlog $q(x)<q(x')$. By density, suppose the difference of $q(x')-q(x)>c>0$ for some $c \in \Gamma$. Also there exists some $c' \in \Gamma$ such that $\lambda -c <q(x)-c'< \lambda$. Hence, we have that $q(c' x) =q(x)-c'<\lambda $ and $q(c'x')=q(x')-c'>q(x)-c'+c>\lambda-c+c=\lambda$. 
    \item If $q(x)=q(x) \in \Gamma$, then $c'=q(x)-\lambda$ will separate the two characters into $[0_+,\lambda_-], [\lambda_+, 1_-]$.
\end{itemize}

\end{proof}
Let us remark that this was observed already in \cite{extensiveamenability}.
Also, it is clear that $\Gamma/\mathbb{Z}$ if finitely generated iff $\Gamma$ is finitely generated. Hence by Theorem \ref{a finitely generated iff} we obtain:
\begin{corollary}
  Let $\Gamma$ be a dense countable subgroup of $\mathbb{R}$ with $1 \in \Gamma$. The derived subgroup of $IE(\Gamma)$ is finitely generated iff $\Gamma$ is finitely generated as a group. \label{finitely generated iff}  
\end{corollary}
The fact that $\Gamma$ finitely generated implies that $D(IE(\Gamma))$ is finitely generated was observed in \cite{extensiveamenability}.
But recall that due to Remark \ref{local embedding}, any of these finitely generated groups in particular embed into polycyclic groups. So we look to focus on the case where $\Gamma \cong \mathbb{Z}^n$ and 
 obtain a finite generating set in this case. Abstractly speaking, Lemma \ref{expansive action IE}  establishes that the dynamical system $\alpha: \Gamma/\mathbb{Z} \acts [0_+,1_-]$ is a subshift in the case when $\Gamma$ is finitely generated. This is a classical result in dynamics, reproven by Nekrashevych in (\cite{nekrashevych2019simple}, Proposition 5.5). Let us describe now how to obtain this picture concretely. \ \\
We are inspired by so-called Sturmian subshifts, a classical object in dynamics. See \cite{blanchard2000topics}, Chapter 1. for a discussion of Sturmian subshifts and the paper \cite{brix2021sturmian} in which Brix studied the associated C$^*$ -algebras. 
Let $\Gamma$ be a dense additive subgroup of $\mathbb{R}$ containing 1. Let $\lambda \in \Gamma \cap (0,1)$ be arbitrary. 

For $t \in [0,1) $, set:
$$x_t: \Gamma/\mathbb{Z} \rightarrow \{0,1\} \quad b \mapsto \begin{cases}
    1 & 0 \leq t+b  < \lambda<1 \ \\
    0 & \text{else} 
\end{cases} $$
Where $t+b$ is taken modulo $\mathbb{Z}$. Let $$X_{\Gamma,\lambda}=\overline{\{x_t \; : \;  t \in [0,1) \}} \subset \{0,1 \}^{\Gamma/\mathbb{Z}}$$
Let $\mathcal{E}_{\Gamma,a}=(X_{\Gamma,a},\sigma)$.
The shift for $c \in \Gamma/\mathbb{Z}$ is given by:
$$\sigma_c: X_\Gamma \rightarrow X_\Gamma \quad   \sigma_c(x)(b)=x(b+c)  $$
Note that $\sigma_c(x_t)=x_{t+c}$ (where addition is taken mod $ \mathbb{Z}$) since
$$\sigma_c(x_t)(b)=x_t(b+c)= \begin{cases}
    1 & 0 \leq t + b + c < \lambda< 1 \ \\ 
    0 &\text{else}
\end{cases} $$
Let us examine further $X_\Gamma$. 
\begin{lemma}
Let $\Gamma$ be a dense countable subgroup of $\mathbb{R}$ with $1 \in \Gamma$. Let $\lambda \in \Gamma \cap (0,1)$ be arbitrary. Then
$X_\Gamma=\{x_t \; : \; t \in [0,1) \}\sqcup \{\hat{x}_t:=\lim_{n \rightarrow \infty} x_{t-1/n} \; : \; t \in (0,1] \cap \Gamma \}  $. The topology is generated by basic open sets of the form:
$$ \overline{\{x_t \; : \; t \in [a,b) \}}=\{x_t \; : \; t \in [a,b) \}\sqcup \{ \hat{x_t} \; : \; t \in (a,b] \} $$
\end{lemma}
\begin{proof}
First let us show that $\{x_t \; : \; t \in [0,1) \}\sqcup \{ \hat{x_t} \; : \; t \in (0,1] \} \subset X_\Gamma $. Consider first $\lim_{n \rightarrow \infty } x_{1-1/n}$. One has that this is convergent (for all $t' \in \Gamma/\mathbb{Z}$ there exists some $N_{t'}\in \mathbb{N}$ such that $n>N_{t'} \implies x_{1-1/n}({t'})=x_{1-1/N_{t'}}(t')$, hence $(\lim_{n \rightarrow \infty } x_{1-1/n})(t')=x_{1-1/N_g}(t')$ ).  At the same time, different from $x_{t'}$ for any $t'$. 

Set $\hat{x}_1:=\lim_{n \rightarrow \infty } x_{1-1/n}$. Suppose for contradiction that $\hat{x}_1=x_t$ for some $t$. Note first that for $n>1/(1-\lambda)$, $x_{1-1/n}(0)=0$. Therefore, $\hat{x}_1(0)=0$, hence $t \in [\lambda,1)$.  Let $c \in (0,1-t) \cap (0,\lambda) $, then $\sigma_c(x_t)(0)=x_{t+c}(0)=0$. However, for any $c \in (0,\lambda)$ we have that $\sigma_c(\hat{x}_1)(0)=\sigma_c(\lim_{n \rightarrow \infty} x_{t-1/n})(0)=\lim_{n \rightarrow \infty }x_{c-1/n} (0)=1$. This is a contradiction hence $\hat{x}_1$ is distinct from all the $x_t, t \in [0,1)$. \ \\
Let $t \in (0,1) \cap \Gamma$. Consider $\hat{x}_t=\lim_{n \rightarrow \infty} x_{t-1/n}$. Since $\sigma_t$ is continuous, it is enough to remark that $\sigma_t(\hat{x}_1)=\sigma_t(\lim_{n \rightarrow \infty} x_{1-1/n})=\lim_{n \rightarrow \infty} \sigma_t(x_{1-1/n})=\lim_{n \rightarrow \infty} x_{t-1/n}=\hat{x}_t$. Therefore $\hat{x}_t$ exists and is distinct from any $x_{t'}$ since we saw that $\{x_{t'} \; : \; t' \in [0,1) \}$ is invariant under $\sigma_t$. 
\ \\ \ \\
Now let us show that $X_\Gamma \subset \{x_t \; : \; t \in [0,1) \}\sqcup \{ \hat{x_t} \; : \; t \in (0,1] \}$. Let $(t_n)_{n \in \mathbb{N}}$ be sequence in $[0,1)$ such that $\lim_{n \rightarrow \infty} x_{t_n}$ exists. Then, for all $t \in \Gamma/\mathbb{Z}$, there exists some $N \in \mathbb{N}$ such that $x_{t_{n+N}}(t) $ is constant $(n \in \mathbb{N})$. This implies that for large $n$, $t_n$ is a convergent sequence in $[0,1)$ that is eventually monotone. If $t_n$ is eventually monotone increasing to $t \in \Gamma$, $\lim_{n} x_{t_n}= \hat{x}_t$. If $t_n$ is eventually monotone decreasing to $t \in \Gamma$, $\lim_{n} x_{t_n}= x_t$, finally if the limit of $t_n$ is some $t \nin \Gamma$, regardless of the direction, $\lim_{n} x_{t_n}=x_t$.  
\ \\ \ \\
Now let us describe the topology on $X_\Gamma$. The topology is generated by cylinder sets. These sets come in two forms $C(a,0):=\{ x \in X_\Gamma \; : \; x(a)=0 \}, C(b,1)=\{x \in X_\Gamma \; x(b)=1 \} $ where $a,b \in \Gamma \cap [0,1)$ are arbitrary. It is clear that $C(a,0)=\overline \{ x_t \; : \; t \in [1-a, 1+\lambda -a) \cup [0,\lambda-a) \}, C(b,1)= \overline \{ x_t \; : t \in [\lambda-b,1-b )\cup [1+\lambda-b,1) \} $. Therefore, for $a<b$, we have that $\overline{\{x_t : t \in [a,b) \}}$ is open, and that cylinder sets can be written $C(a,0):=\overline{\{x_t : t \in [1-a,1+\lambda-a) \}} \cup \overline{\{x_t : t \in [0,\lambda-a) \}}   $, $C(b,1):=\overline{\{x_t : t \in [\lambda-b,1-b) \}} \cup \overline{\{x_t : t \in [1+\lambda-b,1) \}}$. Therefore the topology is generated by the compact open subsets of the form $\overline{\{x_t \; : t \in [a,b)\}}$
\end{proof}
\begin{lemma}
Let $\Gamma$ be a dense countable subgroup of $\mathbb{R}$ with $1 \in \Gamma$. Let $\lambda \in \Gamma \cap (0,1)$. Then, $\mathcal{E}_{\Gamma,\lambda} \cong \Gamma/\mathbb{Z} \ltimes [0_+,1_-]$ (they are conjugate as groupoids). In particular, the groupoid conjugacy class is independent of our choice of $\lambda \in \Gamma \cap (0,1)$. 
    \label{IE as a subshift} 
\end{lemma}
\begin{proof}
 
Our groupoid conjugacy is concrete. 
$$\phi: \mathcal{E}_{\Gamma,\lambda} \rightarrow \Gamma/\mathbb{Z} \ltimes U_{0,1} \quad (c,x)  \mapsto \begin{cases}
   (c, t_+) &  x=x_t, \; t \in \Gamma \ \\
   (c, t) & x=x_t, \; t \nin \Gamma \ \\
    (c,t_- ) & x=\hat{x}_t,  \; t \in \Gamma
\end{cases} $$

Note that in particular, $\phi (\overline{\{x_t \; t \in [c,d) \}})=[c_+,d_-]$ so that $\phi$ restricts to a homeomorphism of the unit spaces.
\end{proof}
Our motivation for this description of the groupoid is the ability to apply results from \cite{chornyi2020topological} to obtain a concrete generating set. First let us recall for a subshift the definition of a patch. 
\begin{definition}
    Let $G$ be an abelian group. Let $\mathcal{A}$ be a finite alphabet and $X \subset \mathcal{A}^G$ be a closed $G$-shift invariant subset. Consider the subshift $(G,X)$. 
    \begin{itemize}
        \item A \textit{patch} is a map $\pi:S \rightarrow \mathcal{A}$ where $\{s_1,...,s_n\}=S \subset G$ is a finite subset. 
        \item With each \textit{patch} we associate a cylinder set $W_\pi:= \{x \in X \; : x|_S=\pi \}$. 
        \item For each patch we say the transformation $T_{\pi}$ is well defined if $W_\pi$ is nonempty and $\{sW_{\pi} \}_{s \in S}$ are pairwise disjoint. 
        \item If $T_{\pi}$ is well defined, then $T_{\pi} \in Homeo(X)$ is the homeomorphism:
        $$ T_{\pi}(x)=\begin{cases} (s_{i+1}-s_{i})x & x \in s_i W_{\pi}, \; i\leq n-1 \ \\
        (s_1-s_{n})x & x \in s_{n} W_{\pi} \ \\
        x & \text{else}\end{cases}$$
        i.e. the homeomorphism of order 3 cyclically permuting the $sW_{\pi}$ in the canonical way. 
    \end{itemize}
\end{definition}
\begin{rmk}
Let $X \subset \{0,1\}^{\mathbb{Z}^d}$ a $\mathbb{Z}^d$ subshift and
 a patch $\pi^i:\{0,e_i,-e_i\} \rightarrow \{0,1\}$, we say $T_{\pi}$ is well defined if $W_\pi, e_i W_{\pi}, -e_i W_{\pi}$ are nonempty and pairwise disjoint. In this case $T_{\pi}$ is the homeomorphism of $X$ given by:
 $$ T_{\pi}(x)=\begin{cases}
     x+e_i & x \in W_{\pi}, -e_i W_{\pi} \ \\
     x-2e_i & x \in e_i W_{\pi} \ \\
     x & \text{ else }
 \end{cases}$$
\end{rmk}
\begin{theorem}[Nekrashevych-Juschenko-Chornyi]
Let $X \subset \{0,1\}^{G}$ be a $G$-subshift, where $G$ is a finitely generated abelian group with generators $e_1,...,e_d$. Then the derived subgroup of the groupoid $\G_{X}$, $D(\G_{X})$ is generated by $T_{\pi^i}$, where $i=1,...,d$ and $\pi^i$ ranges over all $\pi^i$ such that $T_{\pi^i}$ is well defined. 
    \label{nekthm} [\cite{chornyi2020topological}, Proposition 8]
\end{theorem}
Note that as stated in \cite{chornyi2020topological}, the Theorem is only for the case $G=\mathbb{Z}^d$. However, extending this proof for arbitrary finitely generated abelian groups is relatively straightforward, as explained in the proof of [\cite{chornyi2020topological}, Proposition 18], and appears in the PhD thesis of the author in this more general form \cite{mythesis}. 
Applying the above Theorem, we first look to obtain a generating set of $D(IE(\Gamma))$ whenever $\Gamma/\mathbb{Z} \cong \mathbb{Z}^d$, $d>2$. Our first step is to notice we may assume our generators are in a prescribed subset $(a,1/2)$:
\begin{lemma}
Let $\Gamma$ be a dense countable subgroup of $\mathbb{R}$ with $1 \in \Gamma$.   Suppose $\Gamma/\mathbb{Z} \cong \mathbb{Z}^d$, with $d >1$. Then for any $a \in (0,1/2) \cap \mathbb{Q}$, $\Gamma/\mathbb{Z}$ is generated by a set $\{\lambda_i\}_{i=1}^d$, where $a<\lambda_1<\lambda_2<...<\lambda_d<1/2$. \label{prescribed generating set}
\end{lemma}
\begin{proof}
Let us prove this by induction on $d$. If $d=2$, suppose we have a algebraically independent generating set $\hat{\lambda_1},\hat{\lambda_2}$. Then, since $\mathbb{Z} \oplus \lambda_1 \mathbb{Z}$ is dense in $\mathbb{R}$, there exists $n,m \in \mathbb{Z}$ such that $\lambda_2=\hat{\lambda_2}+n+m\hat{\lambda}_1 \in (a,1/2)$. Then it is clear that $\hat{\lambda_1},\lambda_2$ form an algebraically independent generating set for $\Gamma/\mathbb{Z}$. Since $\mathbb{Z} \oplus \lambda_2 \mathbb{Z}$ is dense in $\mathbb{R}$, there exists $n,m \in \mathbb{Z}$ such that 
$\lambda_1=n \lambda_2 +m + \hat{\lambda_1} \in (a,\lambda_2) $. Then again, it is clear that $\lambda_1, \lambda_2$ form an algebraically independent generating set of $\Gamma/\mathbb{Z}$ of the required form. \ \\
Assume the statement is true for $d$ and let us suppose that $\Gamma/\mathbb{Z} \cong \mathbb{Z}^{d+1}$. Then let us assume by the inductive hypothesis that our generators of $\Gamma/\mathbb{Z}$ are of the form $1/3<\lambda_1<...<\lambda_d<1/2$ and $\hat{\lambda}_{d+1}$. Since $\mathbb{Z} \oplus \bigoplus_{i=1}^d\lambda_i \mathbb{Z}$ is dense in $\mathbb{R}$, we may find some $\gamma \in \mathbb{Z} \oplus \bigoplus_{i=1}^d\lambda_i \mathbb{Z}$ such that $\lambda_{d+1}=\gamma + \hat{\lambda}_{d+1} \in (\lambda_d,1/2)$. Then $\{\lambda_i\}_{i=1}^n$ is a generating set of the required form. 
\end{proof}

Let us fix notation for our patch maps, and look to remove patches $\pi$ such that $W_\pi$ is nonempty. 
For $(0,e_i,-e_i)$ let the patch $\pi^i_{a,b,c}$ with $a,b,c \in \{0,1\}$ be the patch such that $\pi^i_{a,b,c}(0)=a, \; \pi^i_{a,b,c}(e_i)=b, \pi^i_{a,b,c}(-e_i)=c$.

\begin{lemma}
  Let $\Gamma$ be a dense countable subgroup of $\mathbb{R}$ with $1 \in \Gamma$.  Suppose further that $\Gamma/\mathbb{Z} \cong \mathbb{Z}^d, d>1$ be generated by $\lambda_i$ with $1/3<\lambda_1<\lambda_2<...<\lambda_d<1/2$. Take $\lambda=\lambda_1$ in our construction of the subshift (Lemma \ref{IE as a subshift}). Then for all $i$ we have that,
    $$ W_{\pi^i_{1,1,0}}=W_{\pi^i_{1,0,1}}=W_{\pi^i_{1,1,1}}=W_{\pi^i_{1,1,0}}=W_{\pi^i_{0,0,0}}= \emptyset$$
  
\end{lemma}
\begin{proof} 
Let us begin with $i$ arbitrary. We show that $W_{\pi^i_{a,b,c}}=\emptyset$ by showing that there exists no $x_t$ such that $x_t \in W_{\pi^i_{a,b,c}}$. Suppose first $x_t \in W_{\pi^i_{1,1,0}}\bigcup W_{\pi^i_{1,1,1}}$. Then in particular $x_t(0)=x_t(\lambda_i)=1$. Hence $t \in [0,\lambda_1)$ and $t+\lambda_i \in [0,\lambda_1)$. This is a contradiction since then 
$\lambda_1 \leq \lambda_i \leq t+\lambda_i <\lambda_1$. Therefore $W_{\pi^i_{1,1,0}}\bigcup W_{\pi^i_{1,1,1}}=\emptyset$. Now suppose $x_t \in W_{\pi^i_{1,0,1}}$. Then in particular $x_t(0)=x_t(-\lambda_i)=1$. It follows that $t,t-\lambda_i \in [0,\lambda_1)$. But then $\lambda_1 \leq (t-\lambda_i)+\lambda_1<t+(-\lambda_i +\lambda_1) \leq t<\lambda_1$. This is a contradiction, hence $W_{\pi^i_{1,0,1}}$ is empty. Now let us suppose $x_t \in W_{\pi^i_{0,0,0}}$. Then we have that $x_t(0)=x_t(\lambda_i)=x_{t}(-\lambda_i)=0$. Then $t,t \pm \lambda_i \in [\lambda_1,1)$. Note in particular then, $\lambda_1<t-\lambda_i<t<t+\lambda_i<1$. This is impossible since $1-\lambda_1<2/3$ but $2\lambda_i>2/3$.
\end{proof}
\begin{lemma}
  Let $\Gamma$ be a dense countable subgroup of $\mathbb{R}$ with $1 \in \Gamma$.  Suppose further that $\Gamma/\mathbb{Z} \cong \mathbb{Z}^d, d>1$ be generated by $\lambda_i$ with $1/3<\lambda_1<\lambda_2<...<\lambda_d<1/2$. Take $\lambda=\lambda_1$ in our construction of the subshift (Lemma \ref{IE as a subshift}).  
Then for all $i$ we have:
\begin{enumerate}
    \item $ x_t \in W_{\pi^i_{1,0,0}} \iff t \in [0,\lambda_1)$
    \item $x_t \in W_{\pi^i_{0,1,0}} \iff t \in [\lambda_1+\lambda_i,1+\lambda_1-\lambda_i)$
    
    \item $x_t \in W_{\pi^i_{0,1,1}} \iff t \in [1-\lambda_i, \lambda_i+\lambda_1)$
    \item $x_t \in W_{\pi^i_{0,0,1}} \iff t \in [\lambda_i,1-\lambda_i)$
\end{enumerate}
\label{intervals lemma}
\end{lemma}

\begin{proof}
    \begin{enumerate}
        \item If $x_t \in W_{\pi^i_{1,0,0}}$, $x_t(0)=1 \implies t \in [0,\lambda_1)$. Moreover, if $t \in [0,\lambda_1)$ then $x_t(0)=0$, $t+\lambda_i \in [\lambda_i,\lambda_1+\lambda_i) \subset [\lambda_1,1)$ hence $x_t(\lambda_i)=0$. $t-\lambda_i \in [1-\lambda_i,1+\lambda_1-\lambda_i) \subset [\lambda_1,1)$ hence $x_t(-\lambda_i)=0$. 
        \item $x_t(0) \iff t \in [\lambda_1,1)$. $x_t(\lambda_i)=1 \iff t \in [1-\lambda_i,1+\lambda_1-\lambda_i)$, $x_t(-\lambda_i)=0 \iff t \in [\lambda_1+\lambda_i,1) $. Note that $1-\lambda_i<2/3<2\lambda_1<\lambda_1+\lambda_i$, hence the intersection of these sets is $[\lambda_1+\lambda_i,1+\lambda_1-\lambda_i)$. 
        \item $x_t(0) \iff t \in [\lambda_1,1)$. $x_t(\lambda_i)=1 \iff t \in [1-\lambda_i,1+\lambda_1-\lambda_i)$, $x_t(-\lambda_i)=1 \iff t \in [\lambda_i,\lambda_1+\lambda_i)$. $\lambda_i<1/2<1-\lambda_i$. Also note that $\lambda_i+\lambda_1<1+\lambda_1-\lambda_i$. Hence, the intersection of these sets is $[1-\lambda_i, \lambda_i+\lambda_1)$. 
    \item
        $x_t(0) \iff t \in [\lambda_1,1)$. $x_t(\lambda_i)=0 \iff t \in [\lambda_1, 1-\lambda_i)$, $x_t(-\lambda_i)=1 \iff t \in [\lambda_i,\lambda_1+\lambda_i)$. Noting that $ 1-\lambda_i<2/3<2\lambda_1\leq \lambda_1+\lambda_i$, it is clear the intersection of these sets is $[\lambda_i,1-\lambda_i)$.
        
    \end{enumerate}
\end{proof}

\begin{lemma}
  Let $\Gamma$ be a dense countable subgroup of $\mathbb{R}$ with $1 \in \Gamma$.  Suppose further that $\Gamma/\mathbb{Z} \cong \mathbb{Z}^d, d>1$ be generated by $\lambda_i$ with $2/5<\lambda_1<\lambda_2<...<\lambda_d<1/2$. Take $\lambda=\lambda_1$ in our construction of the subshift (Lemma \ref{IE as a subshift}). Then for all $i$ the only $a,b,c$ such that $W_{\pi^i_{a,b,c}}$ is nonempty and the sets $W_{\pi^i_{a,b,c}}, \lambda_i W_{\pi^i_{a,b,c}}, -\lambda_i W_{\pi^i_{a,b,c}}$ are pairwise disjoint are $W_{\pi^i_{0,1,0}}$ and $W_{\pi^i_{0,0,1}}$. 
\label{only cylinders of interest}
\end{lemma}
\begin{proof} Note the following equivalences, since intersections of cylinder sets are cylinder sets:
\begin{itemize}
    \item $W_{\pi^i_{a,b,c}}, \pm \lambda_i W_{\pi^i_{a,b,c}}$ are pairwise disjoint. 
    \item There exists no $t$ such that $x_t $ is in the intersection of one of the above sets. 
    \item By Lemma \ref{intervals lemma}, $x_t \in W_{\pi^i_{a,b,c}}$ iff $t \in [d,f)$ for some $d,f \in \Gamma$. It follows that $[d,f), [d+\lambda_i,f+\lambda_i), [d-\lambda_i,f-\lambda_i)$ are disjoint intervals in $[0,1] \mod \mathbb{Z}$.  
\end{itemize}
Note that if the interval in Lemma \ref{intervals lemma} is greater than $1/3$, these three sets cannot be disjoint, since this would imply there are 3 disjoint subintervals of $[0,1]$ all of have Lebesgue measure greater than $1/3$.  Therefore $W_{\pi^i_{1,0,0}}$ can be excluded for all $i$, since we assumed $\lambda_1>1/3$. \ \\ The other intervals are all less than 1/3 in length. In particular then, when we translate by $1/2>\lambda_i>1/3$, we have that $\lambda_i W_{\pi_{a,b,c}^i} \cap W_{\pi_{a,b,c}^i}=\emptyset, -\lambda_i W_{\pi_{a,b,c}^i} \cap W_{\pi_{a,b,c}^i}=\emptyset $. It is for this reason enough to check if $\lambda_i W_{\pi_{a,b,c}^i} \cap -\lambda_i W_{\pi_{a,b,c}^i}=\emptyset$. 
\begin{itemize}
    \item If $(a,b,c)=(0,1,0)$, the sets to compare are $[\lambda_1,1+\lambda_1-\lambda_i)$ and $[\lambda_1+2\lambda_i-1,\lambda_1)$. These sets are disjoint clearly. 
    \item If $(a,b,c)=(0,1,1)$, the sets to compare are $(1-2\lambda_i,\lambda_1]$ and $(0,2\lambda_i+\lambda_1-1]$. These have nonempty intersection $\iff 1-2\lambda_i>2\lambda_i+\lambda_1-1 \iff 2>4 \lambda_i+\lambda_1>5\lambda_1$. We assumed that $\lambda_1>2/5$, hence this never occurs and these sets always have nonempty intersection. 
    \item If $(a,b,c)=(0,0,1)$, the sets to compare are $(0,1-2\lambda_i]$ and $(2\lambda_i,1]$. These sets are disjoint since $1-2\lambda_i<1/3<2/3<2\lambda_i$.
\end{itemize}

\end{proof}
We are now ready to apply Theorem \ref{nekthm} to obtain a generating set.

\begin{theorem}
  Let $\Gamma$ be a dense countable subgroup of $\mathbb{R}$ with $1 \in \Gamma$.  Suppose further that $\Gamma/\mathbb{Z} \cong \mathbb{Z}^d, d>1$. Then there exists a generating set $\{\lambda_i\}_{i=1}^d$ of $\Gamma/\mathbb{Z}$ with $2/5<\lambda_1<\lambda_i<\lambda_{i+1}<\lambda_d<1/2$. Then, $D(IE(\Gamma))$ is generated by 
$$\sigma_i(t)=\begin{cases}
   t+\lambda_i & t \in (\lambda_1,1-2\lambda_i+\lambda_1] \ \\
   t+\lambda_i-1 & t \in (\lambda_1+\lambda_i,1+\lambda_1-\lambda_i] \ \\
    t+1-2\lambda_i & t \in  (\lambda_1+2\lambda_i-1,\lambda_1] \ \\
    t & \text{ else } 
\end{cases}, \quad \hat{\sigma}_i(t)=\begin{cases}
    t+\lambda_i & t \in (0,1-2\lambda_i]\sqcup(\lambda_i,1-\lambda_i] \ \\
    t-2\lambda_i & t \in  (2 \lambda_i,1] \ \\
    t & \text{ else }
\end{cases}$$
\label{ti defined}
Where $i=1,..,d$. 
\label{generating set theorem}
\end{theorem}
\begin{proof}
 First, note that the generating set $\lambda_i$ of $\Gamma/\mathbb{Z}$ occurs via Lemma \ref{prescribed generating set}. We have for all $i$, the only patches $\pi^i:\{0,\lambda_i,-\lambda_i\} \rightarrow \{0,1\}$ such that $T_\pi$ is defined are $\pi^i_{0,1,0}$ and  $\pi^i_{0,0,1}$ by Lemma \ref{only cylinders of interest}. Hence, the result follows by Theorem \ref{nekthm} that $D(IE(\Gamma))$ is generated by $T_{\pi^i_{0,1,0}}, T_{\pi^i_{0,0,1}} $. It is a straightforward computation to check that $T_{1,i}=\varphi \circ \phi ( T_{\pi^i_{0,0,1}} ), \; T_{1,i}=\varphi \circ \phi ( T_{\pi^i_{0,1,0}})$.

\end{proof}

\begin{figure}[h!]
    \centering

\tikzset{every picture/.style={line width=0.75pt}} 

\begin{tikzpicture}[x=0.65pt,y=0.65pt,yscale=-1,xscale=1]

\draw    (410.87,80.79) -- (410.73,300.41) ;
\draw    (410.73,300.41) -- (630.59,300.41) ;
\draw  [dash pattern={on 0.84pt off 2.51pt}]  (394.95,240.24) -- (523.91,241.23) -- (630.59,241.23) ;
\draw  [dash pattern={on 0.84pt off 2.51pt}]  (410.03,225.77) -- (629.9,225.77) ;
\draw  [dash pattern={on 0.84pt off 2.51pt}]  (396.92,166.27) -- (630.24,166.59) ;
\draw  [dash pattern={on 0.84pt off 2.51pt}]  (410.53,139.96) -- (630.39,139.96) ;
\draw  [dash pattern={on 0.84pt off 2.51pt}]  (410.87,80.79) -- (630.73,80.79) ;
\draw  [dash pattern={on 0.84pt off 2.51pt}]  (470.02,81.77) -- (469.88,318.16) ;
\draw  [dash pattern={on 0.84pt off 2.51pt}]  (544.96,80.79) -- (542.84,321.12) ;
\draw  [dash pattern={on 0.84pt off 2.51pt}]  (486.78,79.8) -- (486.64,299.42) ;
\draw  [dash pattern={on 0.84pt off 2.51pt}]  (571.58,80.79) -- (571.47,251.09) -- (571.44,300.41) ;
\draw  [dash pattern={on 0.84pt off 2.51pt}]  (630.73,80.79) -- (630.59,300.41) ;
\draw    (410.03,225.77) -- (468.9,166.27) ;
\draw    (486.09,140.29) -- (544.96,80.79) ;
\draw    (571.44,300.41) -- (630.3,240.91) ;
\draw    (467.32,241.23) -- (485.66,224.46) ;
\draw    (545.21,166.27) -- (571.44,141.61) ;
\draw    (88.87,70.79) -- (88.73,290.41) ;
\draw    (88.73,290.41) -- (308.59,290.41) ;
\draw  [dash pattern={on 0.84pt off 2.51pt}]  (108.73,70.8) -- (107.73,336.8) ;
\draw  [dash pattern={on 0.84pt off 2.51pt}]  (138.73,69.8) -- (139.73,306.8) ;
\draw  [dash pattern={on 0.84pt off 2.51pt}]  (167.73,70.8) -- (167.73,356.8) ;
\draw  [dash pattern={on 0.84pt off 2.51pt}]  (269.73,69.8) -- (270.73,304.8) ;
\draw  [dash pattern={on 0.84pt off 2.51pt}]  (299.73,68.8) -- (297.73,342.8) ;
\draw  [dash pattern={on 0.84pt off 2.51pt}]  (87.75,270.82) -- (307.75,270.78) ;
\draw  [dash pattern={on 0.84pt off 2.51pt}]  (86.74,240.82) -- (308.73,240.8) ;
\draw  [dash pattern={on 0.84pt off 2.51pt}]  (87.74,211.82) -- (309.73,210.8) ;
\draw  [dash pattern={on 0.84pt off 2.51pt}]  (86.72,109.82) -- (309.73,109.8) ;
\draw  [dash pattern={on 0.84pt off 2.51pt}]  (85.72,79.82) -- (309.73,78.8) ;
\draw    (108.73,240.8) -- (137.73,211.8) ;
\draw    (139.73,108.8) -- (168.73,79.8) ;
\draw    (270.73,268.8) -- (299.73,239.8) ;
\draw    (168.73,210.8) -- (269.73,108.8) ;
\draw    (300.73,77.8) -- (309.73,70.8) ;
\draw    (88.73,290.41) -- (108.73,270.8) ;

\draw (443.16,319.61) node [anchor=north west][inner sep=0.75pt]    {$1-2\lambda _{i}$};
\draw (477.91,301.86) node [anchor=north west][inner sep=0.75pt]    {$\lambda _{i}$};
\draw (560.67,302.85) node [anchor=north west][inner sep=0.75pt]    {$2\lambda _{i}$};
\draw (523.08,323.56) node [anchor=north west][inner sep=0.75pt]    {$1-\lambda _{i}$};
\draw (340.62,232.82) node [anchor=north west][inner sep=0.75pt]    {$1-2\lambda _{i}$};
\draw (348.57,156.87) node [anchor=north west][inner sep=0.75pt]    {$1-\lambda _{i}$};
\draw (389.18,214.08) node [anchor=north west][inner sep=0.75pt]    {$\lambda _{i}$};
\draw (383.2,131.23) node [anchor=north west][inner sep=0.75pt]    {$2\lambda _{i}$};
\draw (494,45.4) node [anchor=north west][inner sep=0.75pt]    {$\hat{\sigma_i}$};
\draw (172,35.4) node [anchor=north west][inner sep=0.75pt]    {$\sigma_{i}$};
\draw (71.16,339.61) node [anchor=north west][inner sep=0.75pt]    {$\lambda _{1} +2\lambda _{i} -1$};
\draw (131.91,306.86) node [anchor=north west][inner sep=0.75pt]    {$\lambda _{1}$};
\draw (131.16,362.33) node [anchor=north west][inner sep=0.75pt]    {$\lambda _{1} +1-2\lambda _{i}$};
\draw (241.91,309.86) node [anchor=north west][inner sep=0.75pt]    {$\lambda _{1} +\lambda _{i}$};
\draw (258.91,349.86) node [anchor=north west][inner sep=0.75pt]    {$1+\lambda _{1} -\lambda _{i}$};
\draw (2.91,72.33) node [anchor=north west][inner sep=0.75pt]    {$1+\lambda _{1} -\lambda _{i}$};
\draw (28.91,103.86) node [anchor=north west][inner sep=0.75pt]    {$\lambda _{1} +\lambda _{i}$};
\draw (0.16,203.33) node [anchor=north west][inner sep=0.75pt]    {$\lambda _{1} +1-2\lambda _{i}$};
\draw (63.91,233.86) node [anchor=north west][inner sep=0.75pt]    {$\lambda _{1}$};
\draw (-0.84,261.33) node [anchor=north west][inner sep=0.75pt]    {$\lambda _{1} +2\lambda _{i} -1$};

\end{tikzpicture}

    \caption{A pictoral representation of $\sigma_i,\hat{\sigma}_i$ as defined in Theorem \ref{ti defined}}
    \label{fig:generators for Zd}
\end{figure}
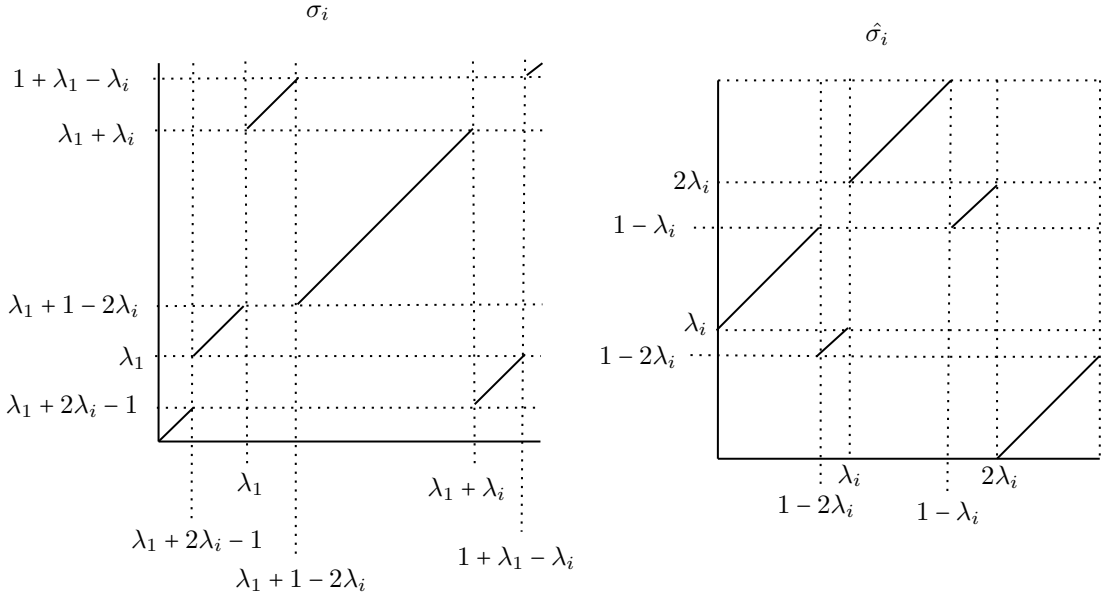
\begin{example}
    \label{concrete rank 2 generating set}
    Let $\Gamma \cong \mathbb{Z}^3$, and suppose $\Gamma \cong \mathbb{Z} \oplus \lambda_1 \mathbb{Z} \oplus \lambda_2 \mathbb{Z}$ (where wlog we assume $2/5<\lambda_i<1/2$ and that $\lambda_i$ are rationally independent). Then the main result of \cite{extensiveamenability} says that $D(IE(\Gamma))$ is a simple finitely generated amenable group. Moreover, we have an explicit generating set given by:
    $$\sigma_1(t)=\begin{cases}
   t+\lambda_1 & t \in (\lambda_1,1-\lambda_1] \ \\
   t+\lambda_1-1 & t \in (2\lambda_1,1] \ \\
    t+1-2\lambda_1 & t \in  (3\lambda_1-1,\lambda_1] \ \\
    t & \text{ else } 
\end{cases}, \quad \hat{\sigma}_1(t)=\begin{cases}
    t+\lambda_1 & t \in (0,1-2\lambda_1]\sqcup(\lambda_1,1-\lambda_1] \ \\
    t-2\lambda_1 & t \in  (2 \lambda_1,1] \ \\
    t & \text{ else }
\end{cases}$$
$$\sigma_2(t)=\begin{cases}
   t+\lambda_2 & t \in (\lambda_1,1-2\lambda_2+\lambda_1] \ \\
   t+\lambda_2-1 & t \in (\lambda_1+\lambda_2,1+\lambda_1-\lambda_2] \ \\
    t+1-2\lambda_2 & t \in  (\lambda_1+2\lambda_2-1,\lambda_1] \ \\
    t & \text{ else } 
\end{cases}, \quad \hat{\sigma}_2(t)=\begin{cases}
    t+\lambda_2 & t \in (0,1-2\lambda_2]\sqcup(\lambda_2,1-\lambda_2] \ \\
    t-2\lambda_2 & t \in  (2 \lambda_2,1] \ \\
    t & \text{ else }
\end{cases}$$
It would be interesting to find independent proof that this group is amenable by using this concrete picture of its generators.
\end{example}

\begin{theorem}
  Let $\Gamma$ be a dense countable subgroup of $\mathbb{R}$ with $1 \in \Gamma$. Suppose $\Gamma \cong \mathbb{Z}^{d+1}$, $d>1$. Suppose further that $\Gamma/\mathbb{Z} \cong \mathbb{Z}^{d}\oplus \mathbb{Z}_k, d>1$, where $k>9$. Using Lemma \ref{prescribed generating set}, we may take the generators of $\Gamma/\mathbb{Z}$ to be $1/k$ and irrational numbers $\lambda_i$, with $2/5<\lambda_1<...<\lambda_d<1/2$. Take $\lambda=\lambda_1$ in our construction of the subshift (Lemma \ref{IE as a subshift}). Then a generating set for $D(IE(\Gamma))$ is given by:
$$ S=\{\sigma_i,\hat{\sigma}_i, r_{k,a} \; : \; i \in \{ 1,2,..,d\}  \; \; a \in \{\lambda_1-1/k, \lambda_1-2/k, 1-2/k,1-1/k\} \}$$
Where $\sigma_i, \hat{\sigma}_i$ are as in Theorem \ref{generating set theorem}, and $r_{k,a}$ is given by:
$$ r_{k,a}(t)=\begin{cases}
    t+1/k \mod{\mathbb{Z}} \; & t \in (a,a+2/k] \ \\
    t-2/k \mod{\mathbb{Z}} \; & t \in (a+2/k,a+3/k]\ \\
    t & \text{ else }
\end{cases}$$
\label{generating set k}
\end{theorem}
\begin{proof}
Let $\pi_{a,b,c}: \{0,1/k,-1/k\} \rightarrow \{0,1\}$ be the patch such that $\pi_{a,b,c}(0)=a, \pi_{a,b,c}(1/k)=b, \pi_{a,b,c}(-1/k)=c$. As before we fix the picture of this dynamical as a subshift where $x_t(0)=1 \iff t \in [0,\lambda_1)$.  
\begin{itemize}
    \item $x_t \in \sqcup_{b,c \in \{0,1\}} W_{\pi_{1,b,c}} \iff t \in [0,\lambda_1)$
    \item $x_t \in \sqcup_{a,c \in \{0,1\}} W_{\pi_{a,1,c}} \iff t \in [0,\lambda_1-1/k)\sqcup [1-1/k,1)$
    \item $x_t \in \sqcup_{a,b \in \{0,1\}} W_{\pi_{a,b,1}} \iff t \in [1/k,1/k + \lambda_1)$
\end{itemize}
Therefore
$\emptyset= W_{\pi_{1,0,0}}=W_{\pi_{0,1,1}}$. Also we have that
Note some of the translated cylinder sets are also not disjoint--
$x_{1/k} \in W_{\pi_{1,1,1}} \cap -1/k W_{\pi_{1,1,1}}$  and
$x_{1-2/k} \in W_{\pi_{0,0,0}} \cap +1/k W_{\pi_{0,0,0}}$
The other corresponding $T_{\pi_{a,b,c}}$ are well defined. It is straightforward to verify that:
$$ \varphi \circ \phi (T_{\pi_{0,1,0}})=r_{k, 1-2/k }, \quad \varphi \circ \phi (T_{\pi_{1,1,0}})=r_{k, 1-1/k },$$
$$\varphi \circ \phi (T_{\pi_{1,0,1}})=r_{k, \lambda_1-2/k }, \quad \varphi \circ \phi (T_{\pi_{0,0,1}})=r_{k, \lambda_1-1/k }.$$

Then, the proof follows via Theorem \ref{generating set theorem} and Theorem \ref{nekthm}. 
\end{proof}
Let us remark how these generating sets behave with respect to inclusions $\Gamma' \subset \Gamma$:
\begin{corollary}
    \label{nested generating sets}
    Let $\Gamma' \subset \Gamma$ be dense, nested cyclic subgroups, finitely generated subgroups of $\mathbb{R}$ containing 1. Given a subset $P \subset \mathbb{R}$, Let $Rank_\mathbb{Q}(P)$ denote the number of rationally independent numbers in $P$. Then there exists generating sets $S,S'$ of $\Gamma$, $\Gamma'$ such that:
    \begin{itemize}
        \item $S' \subset S$
        \item $|S \setminus S'| \leq  2Rank_{\mathbb{Q}}(\Gamma'\setminus \Gamma)+4 $. 
        \item $|S| \leq 2 Rank_\mathbb{Q}(\Gamma) +4 $
    \end{itemize}

\end{corollary}
Let us remark that IET locally embeds to groups of the form $\Gamma=1/k \mathbb{Z}\oplus \bigoplus_{i=1}^d \lambda_i \mathbb{Z}$, where $d>1, k>9$. This is just Remark \ref{local embedding}, and letting $k$ being arbitrarily large enough to contain all the rational angles of a given finite subset of IET. Note also that $D(IE(\Gamma))$ is amenable iff $IE(\Gamma)$ is amenable. Hence, we have the following reductions of Question \ref{question intro}:
\begin{corollary}
The following are equivalent:
\begin{itemize}
    \item IET is nonamenable. 
    \item There exists $k>9$, and a finite collection of irrational numbers $\lambda_1,...,\lambda_d \in [2/5,1/2)$ such that the group generated by $S$ as in Theorem \ref{generating set k} is nonamenable.
\end{itemize}
\end{corollary}
And similarly for the existence of a nonamenable free subgroup. 

\section{Homology of Interval Exchange Groups}

Recently there have been impressive developments in the understanding of groupoid homology, and the connections to K-Theory of the reduced C*-algebra (e.g. \cite{bonicke2021dynamic}) and the group homology of Topological full groups \cite{li2022}. We look to apply such results to further our understanding of the interval exchange groups. 
Let us first compute the homology of the underlying groupoids. Let us remark that as with all globalisable partial actions we may take via 
[\cite{matui2012homology}, Theorem 3.6] that:
   $$ H_*([0_+,1_-]\ltimes_\alpha \Gamma ) \cong H_*( \Gamma \ltimes \mathbb{R}_\Gamma |_{[0_+1_-]}^{[0_+1_-]})\cong H_*(\Gamma \ltimes \mathbb{R}_\Gamma) \cong H_*(\Gamma,C_c(\mathbb{R}_\Gamma,\mathbb{Z}))$$
In other words, the groupoid homology with coefficients in $\mathbb{Z}$ reduces to the group homology with coefficients in $C_c(\mathbb{R}_\Gamma,\mathbb{Z})$. It is for this reason we did not introduce groupoid homology. In fact, in our case the groupoid homology is just a shifted version of the homology of $\Gamma$:
\begin{lemma}
  Let $\Gamma$ be a dense additive subgroup of $\mathbb{R}$ containing $1$. Then, $k \geq 0$ $H_k(C([0_+,1_-])\ltimes_\alpha \Gamma ) \cong H_{k+1}(\Gamma)$ \label{gpd hom}
\end{lemma}
\begin{proof}

We follow a similar technique to that used in (\cite{xinlambda}, Proposition 5.8).
     Let $\hat{\mathbb{R}}_\Gamma=\mathbb{R}_\Gamma \bigcup \{+\infty \}$ be the one sided compactification of $\mathbb{R}_\Gamma$, i.e. the basic open sets in $\hat{\mathbb{R}}_\Gamma$ are $[a_+,b_-], [a_+,\infty], a,b \in \Gamma$ . Let $\Gamma \acts \hat{\mathbb{R}}_\Gamma $ by acting as $\alpha$ on $\mathbb{R}_\Gamma$ and by fixing $\infty$. 
     
    Because $\{1_{[a_+, \infty]} \; : a \in \Gamma \}$ forms a $\mathbb{Z}$-basis of $C_c(\hat{\mathbb{R}}_\Gamma,\mathbb{Z})$ that $\Gamma$ acts freely and transitively on. Hence, $C_c(\hat{\mathbb{R}}_\Gamma,\mathbb{Z}) \cong \mathbb{Z}\Gamma$ as $\mathbb{Z}\Gamma$ modules. 
     $$ H_*(\Gamma,C_c(\hat{\mathbb{R}}_\Gamma,\mathbb{Z})) \cong \begin{cases} \mathbb{Z} & *=0 \ \\ 0 & \text{else} \end{cases}.$$
     Then note that 
     $$0 \rightarrow C_c(\mathbb{R}_\Gamma, \mathbb{Z}) \rightarrow C_c (\hat{\mathbb{R}}_\Gamma, \mathbb{Z}) \rightarrow \mathbb{Z} \rightarrow 0 $$
Is a short exact sequence of $\mathbb{Z} \Gamma$ modules. By Proposition 6.1, Chapter III of Brown \cite{brown2012cohomology}, we get that there is a long exact sequence:
$$...\rightarrow H_1(\Gamma, C_c(\hat{\mathbb{R}}_\Gamma,\mathbb{Z}))\rightarrow H_1(\Gamma) \rightarrow H_0(\Gamma,C_c(\mathbb{R}_\Gamma,\mathbb{Z})) \rightarrow  H_0(\Gamma, C_c(\hat{\mathbb{R}}_\Gamma,\mathbb{Z}))\rightarrow H_0(\Gamma) \rightarrow 0$$
But we already know that $H_1(\Gamma)=\Gamma_{ab}=\Gamma, H_0(\Gamma)=\mathbb{Z}$, $H_0(\Gamma, C_c(\hat{\mathbb{R}}_\Gamma,\mathbb{Z}))=\mathbb{Z}$ and for $H_i(\Gamma, C_c(\hat{\mathbb{R}}_\Gamma,\mathbb{Z}))=0$ for $i \geq 1$. Hence we get that there is a long exact sequence of the form:
$$ \rightarrow 0 \rightarrow  H_{i+1}(\Gamma) \rightarrow H_i(C([0_+,1_-])\ltimes_\alpha \Gamma ) \rightarrow 0 \rightarrow ...  $$
$$...  \rightarrow 0 \rightarrow \Gamma \rightarrow H_0(\Gamma,C_c(\mathbb{R}_\Gamma,\mathbb{Z})) \rightarrow  \mathbb{Z} \rightarrow \mathbb{Z} \rightarrow 0 $$
The section around $H_i(C([0_+,1_-])\ltimes_\alpha \Gamma )$ looks like this for  all $i \geq 0$, forcing an isomorphism $ H_{i+1}(\Gamma) \cong H_i(C([0_+,1_-])\ltimes_\alpha \Gamma )$. Around $H_0(\Gamma,C_c(\mathbb{R}_\Gamma,\mathbb{Z}))$, exactness forces the map $H_0(\Gamma,C_c(\mathbb{R}_\Gamma,\mathbb{Z}) \rightarrow \mathbb{Z}$ to be $0$: 
$$H_0(\Gamma,C_c(\mathbb{R}_\Gamma,\mathbb{Z}) \xrightarrow{0} \mathbb{Z} \xrightarrow{\cong} \mathbb{Z} \xrightarrow{0} 0 $$ hence, there is also an isomorphism $\Gamma \cong H_0(\Gamma,C_c(\mathbb{R}_\Gamma,\mathbb{Z})$.  The conclusion follows. 
\end{proof}
Note in particular that $H_0([0_+,1_-] \ltimes_\alpha \Gamma)=H_1(\Gamma)=\Gamma_{ab}=\Gamma$. 
\begin{corollary}
Let $\Gamma$ be a dense additive subgroup of $\mathbb{R}$ containing $1$. Then, $H_0([0_+,1_-] \ltimes_\alpha \Gamma) \cong \Gamma$
\end{corollary}
Note that for this class of groupoids, the HK conjecture holds, that is the K -theory of reduced groupoid C$^*$ -algebra can be related to the groupoid homology; the HK Conjecture holds for $[0_+ ,1_- ] \ltimes_\alpha \Gamma$.
\begin{theorem}
Let $\Gamma$ be a dense countable subgroup of $\mathbb{R}$ containing 1. 
Then, the HK conjecture holds for $[0_+,1_-] \ltimes_\alpha \Gamma$, i.e.
$$K_0(C_r^*(C([0_+,1_-])\ltimes_\alpha \Gamma )) \cong K_1(C_r^*(\Gamma)) \cong \bigoplus_{i =0}^\infty H_{2i}([0_+,1_-] \ltimes_\alpha \Gamma)\cong \bigoplus_{i =0}^\infty H_{2i+1}(\Gamma) $$ 
$$K_1(C_r^*(C([0_+,1_-])\ltimes_\alpha \Gamma )) \cong K_0(C_r^*(\Gamma))/\mathbb{Z}[1]_0 \cong \bigoplus_{i =0}^\infty H_{2i+1}([0_+,1_-] \ltimes_\alpha \Gamma)\cong \bigoplus_{i =1}^\infty H_{2i}(\Gamma) $$ \label{hk conjecture}
\end{theorem}
\begin{proof}
    This was observed in the case where $\Gamma$ is a ring in \cite{xinlambda}, Remark 5.8. One can verify the case where $\Gamma$ is polycyclic by direct computation. This is in fact enough, since we can write any $\Gamma$ as an inductive limit of polycyclic groups. Set $\Gamma=\text{lim}_{n \rightarrow}\Gamma_n$, where $\Gamma_n$ is polycyclic. It follows that $K_i(C_r^*([0_+,1_-] \ltimes_\alpha \Gamma)) \cong \text{lim}_{n \rightarrow} K_i(C_r^*([0_+,1_-] \ltimes_\alpha \Gamma_n)),   H_i([0_+,1_-] \ltimes_\alpha \Gamma) \cong \text{lim}_{n \rightarrow}  H_i([0_+,1_-] \ltimes_\alpha \Gamma_n)$; both the K-Theory and Groupoid homology decompose as direct limits of this case.
\end{proof} We would also like to understand the homology of $IE(\Gamma)$ in terms of the homology of $\Gamma$. Let us construct certain maps $I,j$ appearing in the AH conjecture. 
Let $B$ be a compact open bisection of $\Gamma \ltimes_\alpha [0_+,1_-]$ with $s(B) \cap r(B) =\emptyset$.   $$j: \Gamma \otimes \mathbb{Z}_2 \cong H_0(\Gamma \ltimes_\alpha [0_+,1_-]) \otimes \mathbb{Z}_2 \rightarrow IE(\Gamma)_{ab} \quad [1_{s(B)}]_{H_0} \rightarrow \gamma_B$$
Where $\gamma_B$ is the generator of the symmetric group, as in \ref{def symmetric}. 
$$I: IE(\Gamma)_{ab} \rightarrow H_1(\Gamma \ltimes \alpha)\cong H_2(\Gamma) \quad \hat{B} \mapsto [1_B]_{H_1}$$

Following \cite{li2022}, we can also determine some information about the homology of $IE(\Gamma)$ in terms of the homology of $[0_+,1_-] \ltimes_\alpha \Gamma$ using $I,j$. First we note that since $[0_+,1_-] \ltimes_\alpha \Gamma$ is a minimal, almost finite groupoid, with a Cantor unit space, [\cite{li2022}, Corollary E] allows us to describe the abelianization of these groups in terms of groupoid homology via the so-called AH conjecture. Substituting what we know, we have that:

\begin{lemma}[\cite{li2022}, Corollary E] Let $\Gamma$ be a dense additive subgroup of $\mathbb{R}$ containing $1$. Then,
 $I$ and $j$ are well defined and moreover there exists a long exact sequence:
    $$H_2(\mathsf{D}(IE(\Gamma)) \rightarrow H_3(\Gamma) \rightarrow \Gamma \otimes \mathbb{Z}_2 \xrightarrow{j} IE(\Gamma)_{ab}\xrightarrow{I} H_2(\Gamma) \rightarrow 0   $$ \label{ah conj}
\end{lemma}

\begin{theorem}
    \label{ie finitely generated}
Let $\Gamma$ be a dense countable subgroup of $\mathbb{R}$ containing $1$. The following are equivalent:
    \begin{itemize}
        \item $\Gamma$ is finitely generated. 
        \item $D(IE(\Gamma))$ is finitely generated. 
        \item $IE(\Gamma)$ is finitely generated. 
    \end{itemize}
\end{theorem}
\begin{proof}
We also showed already that $\Gamma$ is finitely generated iff $D(IE(\Gamma))$ is finitely generated (Corollary \ref{a finitely generated iff}). If $\Gamma$ is finitely generated, then $H_2(\Gamma)$ is also finitely generated. By extension, $IE(\Gamma)_{ab}$ is finitely generated  by Theorem \ref{ah conj}.  It follows that $IE(\Gamma)$, the extension of $IE(\Gamma)_{ab}$ by $D(IE(\Gamma))$ is finitely generated whenever $\Gamma$ is finitely generated. 

\ \\
  Now suppose $IE(\Gamma)$ is finitely generated. Let $f_1,...,f_n \in IE(\Gamma)$ be the generators. Let $A$ be the union of all the angles in each $f_i$. We then have that $\mathbb{Z}A \subset \Gamma$ since each $f_i$ in $IE(\Gamma)$, conversely, every $g \in IE(\Gamma)$ is a finite string of $f_i$, hence $\Gamma \subset \mathbb{Z}A$. It follows that $A$ is a finite generating set of $\Gamma$.  
\end{proof}

Notice that we can describe the rational homology of $IE(\Gamma)$ in terms of the rational homology of $\Gamma$ by applying [\cite{li2022}, Corollary C]. 
\begin{lemma} \label{rational homology}
Let $\Gamma$ be a dense additive subgroup of $\mathbb{R}$ containing $1$. For a group $G$, let $$H_{*}^{even}(G)=\begin{cases}
    H_*(G) & *  \text{ even } \ \\
    \{0\} & \text{ else }
\end{cases} \quad H_*^{odd}(G)=\begin{cases}
    H_*(G) & * \text{ odd } \ \\
    \{0\} & \text{ else }
\end{cases} $$ 
and, let 
$$H_{*>2}^{even}(G)=\begin{cases}
    H_*(G) & *>2 \text{ even } \ \\
    \{0\} & \text{ else }
\end{cases} $$ 
Then,
$$H_*(IE(\Gamma), \mathbb{Q})\cong Ext(H_{*+1}^{even}(\Gamma, \mathbb{Q})) \otimes Sym(H_{*+1}^{odd}(\Gamma,\mathbb{Q}))$$
$$H_*(D(IE(\Gamma)), \mathbb{Q})\cong Ext(H_{*+1>2}^{even}(\Gamma, \mathbb{Q})) \otimes Sym(H_{*+1}^{odd}(\Gamma,\mathbb{Q}))$$
Where $Ext,Sym$ denote respectively the Exterior and Symmetric algebras in the sense of Multilinear Algebra \cite{greub1978multilinear}. 
\end{lemma}
\begin{corollary}
   Let $\Gamma,\Gamma'$ be dense additive subgroups of $\mathbb{R}$ containing 1. Then, $H_*(\Gamma, \mathbb{Q}) \cong H_*(\Gamma, \mathbb{Q}) \implies H_*(IE(\Gamma), \mathbb{Q})\cong H_*(IE(\Gamma), \mathbb{Q})$
\end{corollary}
We can also say the following about integral homological stability by applying [\cite{xinlambda}, Theorem F].
\begin{lemma}
Let $\Gamma$ be a dense countable subgroup of $\mathbb{R}$ containing $1$ and some $x \neq 0$. Then, the isomorphism $\Gamma \rightarrow 1/x \Gamma \quad t \mapsto t/x$ induces an isomorphism in homology $H_*(IE(\Gamma)) \cong H_*(IE(1/x \Gamma))$ \label{homological stability}
\end{lemma}\begin{proof}
    Assume wlog that $x>1$. Then, it is clear that $IE(\Gamma)$ is isomorphic to the subgroup $IE(1/x \Gamma)_{[0,1/x]}=\{f \in IE(1/x \Gamma) \; : \; t>1/x \implies f(t)=t \}$ by conjugating via the homothety $t \mapsto t/x$. Also, $IE(1/x \Gamma)_{[0,1/x]}=\mathsf{F}(1/x \Gamma \ltimes_\alpha [0_+,1_-]|_{[0_+,1/x_-]}^{[0_+,1/x_-]})$. Therefore, the result follows via [\cite{xinlambda}, Theorem  F]. 
\end{proof}

\section{Concrete Examples of Interval Exchange Groups}
\subsection{Rational $\Gamma$ and UHF Groupoids}
\label{subs rational}
An easy case for computing an explicit generating set is when $\Gamma \subset \mathbb{Q}$, 
by identifying $IE(\Gamma)$ with a certain increasing sequence of symmetric groups. 
\begin{lemma}
Let $\Gamma \subset \mathbb{Q}$ be dense with $1 \in \Gamma$ . Then $\Gamma$ is generated by a strictly decreasing sequence of rational numbers $\{1/k_i\}_{i=1}^{\infty} \subset \mathbb{Q}$. Moreover, set $k(n)=\prod_{i=1}^n k_i$. Then, $IE(\Gamma) \cong \lim_{n \rightarrow \infty} S_{k(n)}$. The connecting map is given by (letting $\sigma_{i,j}^n$ be the element of $S_{k(n)}$ permuting the $i$ and $j$th element):
$$\iota_n: S_{k(n)} \hookrightarrow S_{k(n+1)} \quad  \sigma_{i,i+1}^n \mapsto \prod_{k=1}^{k_n-1} \sigma_{i k_n +k ,(i+1) k_n+k}^{n+1}$$ 
An explicit generating set is given by:
$$\{ \sigma_{i,i+1}^n \; : 0<i< k(n), n \in \mathbb{N} \} $$
Subject to the (infinitely many) relations:
\begin{itemize}
    \item $ \sigma_{i,i+1}^n=\prod_{k=1}^{k_n-1} \sigma_{i k_n +k ,(i+1) k_n+k}^{n+1}, \; 1 \leq i<k(n) $
    \item $(\sigma_{i,i+1}^n)^2=1 , \; 1 \leq i<k(n)$
    \item $1=[\sigma_{i,i+1},\sigma_{j,j+1}] \;  1\leq i <i+2 \leq j <k(n)$
    \item $\sigma_{i,i+1}^n \sigma_{i+1,i+2}^n \sigma_{i,i+1}^n =\sigma_{i+1,i+2}^n \sigma_{i,i+1}^n \sigma_{i+1,i+2}^n, , \; 1 \leq i<k(n)-1  $
\end{itemize} In terms of piecewise linear bijections, one can take $\sigma_{i,j}^n(t)=\begin{cases}  t+ \frac{j-i}{k(n)} & t \in (\frac{i}{k(n)}, \frac{i+1}{k(n)}] \ \\
 t+ \frac{i-j}{k(n)} & t \in (\frac{j}{k(n)}, \frac{j+1}{k(n)}]  \ \\
 t & \text{else}

 \end{cases}$
 \label{ generating set when rational }
\end{lemma}
\begin{proof}
 If $\Gamma \subset \mathbb{Q}$, it is clear that $\Gamma$ is generated by the (strictly decreasing) sequence of rational numbers $\{1/k_i\}_{i=1}^{\infty}$. Let $\Gamma_n=1/(k(n))\mathbb{Z}$. It is clear that $\Gamma=\bigcup \Gamma_n$. Then note we have:
 $$IE(\Gamma)=[[\Gamma/\mathbb{Z} \ltimes U_{0,1}]]=\bigcup_{n \in \mathbb{N}} [[ \Gamma_n /\mathbb{Z} \ltimes U_{0,1}]] \cong \bigcup_{n \in \mathbb{N}} S_{k(n)}$$
 And that the generators of $[[ \Gamma_n /\mathbb{Z} \ltimes U_{0,1} ]]$ are $\sigma_{i,j}^n$
\end{proof}

Implicitly in the above proof we also realize independently $IE(\Gamma)$ for $\Gamma \subset \mathbb{Q}$ cannot be finitely generated. We also verify that $IE(\Gamma)$ is amenable in this case. Via Theorem \ref{matui isomorphism theorem}, we also obtain an alternative proof of Corollary \ref{identification with UHF}. \ \\

This Corollary establishes that $[0_+,1_-] \ltimes_{\alpha} \Gamma$ in this case is the unique AF groupoid associated with the following Bratelli diagram:
$$ 1 \xrightarrow{k_1} k_1 \xrightarrow{k_2} k(2) \xrightarrow{k_3} k(3) \xrightarrow{k_4} ... \xrightarrow{k_n} k(n) \xrightarrow{k_{n+1}} ... $$
With each $k_i>1$. The supernatural number associated with this AF algebra is always infinite and given by $\prod_{i=1}^\infty k_i=\lim_{n \rightarrow \infty} k(n)$. 

\ \\If $\Gamma$ is generated by a strictly decreasing sequence of rational numbers $\{1/k_i\}_{i=1}^\infty$, then $\Gamma$ is the direct limit of the sequence: 
    $ \mathbb{Z} \xrightarrow{k_1} \mathbb{Z}  \xrightarrow{k_2} \mathbb{Z}  \xrightarrow{k_3} ...$. From this one can observe $H_0(\Gamma)=\Gamma$ and $H_*(\Gamma)=0$ for all $* \geq 1$.  Therefore, via Lemma \ref{gpd hom}
    we obtain that 
    $$ H_*(\Gamma \ltimes [0_+,1_-]) \cong H_{*+1}(\Gamma)=\begin{cases}
        \Gamma & *=0 \ \\
        0 & \text{else} 
    \end{cases}$$
 We have then that Lemma \ref{ah conj} collapses to an isomorphism $$IE(\Gamma)_{ab} \cong \Gamma \otimes \mathbb{Z}_2 \cong \mathbb{Z}_2.$$
 This demonstrates that $ IE(\Gamma) \cong \mathsf{S}( \Gamma \ltimes [0_-,1_+])$ since the map $j$ in Lemma \ref{ah conj} is the zero map. This can also be verified via the generating set in Lemma \ref{ generating set when rational }, since it is clear that each $\sigma_{i,j}^n \in  \mathsf{S}( \Gamma \ltimes [0_-,1_+])$. There is also only one nontrivial quotient homomorphism, given by the sign homomorphism:
 $$sgn:  IE(\Gamma) \rightarrow (\mathbb{Z}_2,+) \quad \sigma_{i,j}^n \mapsto 1 $$
As with the classical symmetric group on $n$-generators. Let us also apply Lemma \ref{rational homology}. We have that $D(IE(\Gamma))$ is rationally acyclic immediately. Noting then $IE(\Gamma)$ is the extension of a rationally acyclic group by a rationally acyclic group:
$$0 \rightarrow D(IE(\Gamma)) \xrightarrow{\iota} IE(\Gamma) \xrightarrow{sgn} \mathbb{Z}_2 \rightarrow 0 $$
$IE(\Gamma)$ is also rationally acyclic.
\begin{example}[Measure Preserving $V$ and the CAR algebra]
    Let $n \in \mathbb{N}$, $n \geq 2$ and $\Gamma=\mathbb{Z}[1/n] \subset \mathbb{Q}$ be the ring generated by $1/n$. Then, $IE(\mathbb{Z}[1/n])$ is canonically isomorphic to the group of the measure-preserving subgroup of the Higman-Thompson group $V_{n,1}$. The associated Bratelli diagram is given by:
    $$ 1 \xrightarrow{n} n  \xrightarrow{n} n^2  \xrightarrow{n} n^3  \xrightarrow{n} ... $$
   And the corresponding UHF algebra is $\bigcup_{k=1}^\infty M_{n^k}(\mathbb{C})$, where the inclusion maps are given by block diagonal inclusions: $$ A_{n^k} \mapsto \begin{pmatrix}
       A_{n^k}^{(1)} & 0_{n^k} & \dots & 0_{n^k} \ \\ 
       0_{n^k} & A_{n^k}^{(2)} & \dots & 0_{n^k} \ \\
       \vdots & \vdots & \ddots & \vdots  \ \\
       0_{n^k} & 0_{n^k} & \dots & A_{n^k}^{(n)}
   \end{pmatrix}$$
   
   Moreover, if $n=2$, the groupoid $\mathbb{Z}[1/2] \ltimes [0_+,1_-]$ is isomorphic to the standard groupoid model of the CAR algebra with supernatural number $2^\infty$. 
\end{example}
\begin{example}[$\mathbb{Q}$ and the universal UHF algebra $\mathcal{Q}$] 
 A final example to highlight is that of $\mathbb{Q}$ and the universal UHF algebra $\mathcal{Q}$. Here, we are discussing the groupoid associated with the following Bratelli diagram:
 $$1 \xrightarrow{2} 2 \xrightarrow{3} 6 \xrightarrow{4} ... \xrightarrow{n} n! \xrightarrow{n+1} ... $$
 This groupoid is the standard model of the universal UHF algebra $\mathcal{Q}$ and $IE(\mathbb{Q})=\bigcup_{n \in \mathbb{N}} S_{n!}$. 
\end{example}

\subsection{$\Gamma$ Polycyclic}
\label{subs polycyclic}
Now let us suppose that we are in the case where $\Gamma$ is polycyclic and generated by irrational numbers that is, let $\Gamma=\mathbb{Z} \oplus \bigoplus_{i=1}^d \lambda_i \mathbb{Z}$ where $\lambda_i$ are $\mathbb{Q}$-independent numbers and $d \geq 1$.  As an abstract group $\Gamma \cong \mathbb{Z}^{d+1}$. Therefore, we have that:
$$ H_*(\Gamma \ltimes [0_+,1_-]) \cong H_{*+1}(\mathbb{Z}^{d+1}) \cong \mathbb{Z}^{ d+1 \choose *+1 }  $$
For example, if $d=1$ we have $H_0 (\mathbb{Z} \oplus \lambda_1 \mathbb{Z} \ltimes [0_+,1_-]) \cong \mathbb{Z}^2$, $H_1(\mathbb{Z} \oplus \lambda_1 \mathbb{Z} \ltimes [0_+,1_-] )\cong \mathbb{Z}$. otherwise the homology vanishes. Plugging this into Lemma \ref{ah conj} we obtain the short exact sequence:
$$0 \rightarrow \mathbb{Z}^2 \otimes \mathbb{Z}_2=\mathbb{Z}_2^2 \xrightarrow{j} IE((\mathbb{Z} \oplus \lambda_1 \mathbb{Z})_{ab} \xrightarrow{I} \mathbb{Z} \rightarrow 0 $$
It follows that $IE(\mathbb{Z} \oplus \lambda_1 \mathbb{Z})_{ab} \cong \mathbb{Z} \oplus \mathbb{Z}_2^2$. 
\ \\ \ \\
Rationally then, $H_1(IE((\mathbb{Z} \oplus \lambda_1 \mathbb{Z}), \mathbb{Q})\cong H_0(IE((\mathbb{Z} \oplus \lambda_1 \mathbb{Z}),\mathbb{Q}) \cong \mathbb{Q} $, and for $* \geq 1, H_*(IE((\mathbb{Z} \oplus \lambda_1 \mathbb{Z}), \mathbb{Q})=0$ by Lemma \ref{rational homology}. In other words, when $d=1$, we can see that $D(IE((\mathbb{Z} \oplus \lambda_1 \mathbb{Z}))$ is rationally acyclic through Lemma \ref{rational homology}. \ \\ \ \\
This is contrast to the case $d>1$. If $d>1$ one has that $H_{d}(\mathbb{Z} \oplus \bigoplus_{i=1}^d  \lambda_i \mathbb{Z}\ltimes [0_-,1_+]) \cong \mathbb{Z}$ and for all $*>d, H_{d}(\mathbb{Z} \oplus \bigoplus_{i=1}^d  \lambda_i \mathbb{Z} \ltimes [0_-,1_+]) \cong 0$. This implies  $ H_{d}(D(IE(\mathbb{Z} \oplus \bigoplus_{i=1}^d  \lambda_i \mathbb{Z}),\mathbb{Q}) \neq 0 $ by Lemma \ref{rational homology}.
\ \\ \ \\
Suppose $d=2$. Then $H_2(\mathbb{Z} \oplus \lambda_1\mathbb{Z} \oplus \lambda_2 \mathbb{Z} \ltimes [0_+, 1_-]) \cong \mathbb{Z}$ $H_0(\mathbb{Z} \oplus \lambda_1\mathbb{Z} \oplus \lambda_2 \mathbb{Z} \ltimes [0_+, 1_-]) \cong H_1(\mathbb{Z} \oplus \lambda_1\mathbb{Z} \oplus \lambda_2 \mathbb{Z} \ltimes [0_+, 1_-]) \cong \mathbb{Z}^3$ and all other homology terms vanish. Via Lemma \ref{ah conj} we obtain a short exact sequence: 
$$  0 \rightarrow \mathbb{Z}_2^3 \xrightarrow{j} IE(\mathbb{Z} \oplus \lambda_1\mathbb{Z} \oplus \lambda_2 \mathbb{Z})_{ab} \xrightarrow{I} \mathbb{Z} \rightarrow 0$$
We again have that this short exact sequence splits and that $IE(\mathbb{Z} \oplus \lambda_1\mathbb{Z} \oplus \lambda_2 \mathbb{Z})_{ab} \cong \mathbb{Z} \oplus \mathbb{Z}_2^3$. 
\ \\ \ \\
If $d > 2$, $H_3(\mathbb{Z} \oplus \bigoplus_{i=1}^d  \lambda_i \mathbb{Z} ) \neq 0$, and so we do not get a short exact sequence in Lemma \ref{ah conj}. However, we may still obtain some information from the long exact sequence, which shall look like:
$$ H_2(D(IE(\mathbb{Z} \oplus \bigoplus_{i=1}^d  \lambda_i \mathbb{Z})) \rightarrow \mathbb{Z}^{d+1 \choose 4} \rightarrow \mathbb{Z}^{d+1} \otimes \mathbb{Z}_2 \xrightarrow{j} IE(\mathbb{Z} \oplus \bigoplus_{i=1}^d  \lambda_i \mathbb{Z})_{ab} \xrightarrow{I} \mathbb{Z}^{d+1 \choose 3} \rightarrow 0$$
We need a more concrete description of $I,j$. We have established already that $\mathsf{D}(IE(\Gamma)) \cong \mathsf{A}(\Gamma \ltimes [0_+ , 1_-])$ (Corollary \ref{derived ie is simple}). Then, \cite{li2022}, Corollary 6.17 identifies $Ker(I)=Im(j)$ with $\mathsf{S}(\Gamma \ltimes [0_+ , 1_-])$. This can also be verified explicitly, if $\gamma_B$ is a generator of $\mathsf{S}(\Gamma \ltimes [0_+ , 1_-])$ then, $j([1_{s(B)}])=\gamma_B$. This demonstrates that $\mathsf{S}(\Gamma \ltimes [0_+ , 1_-])$ is at most $d+1$ generated.
\ \\ \ \\
Note that we have only considered the case where $\Gamma/\mathbb{Z}$ is torsion free and $\Gamma$ is polycyclic. The reason this is enough is that for any $\Gamma $ polycyclic with $\Gamma/\mathbb{Z}$ having torsion, there exists some $k \in \mathbb{N}$ such that $k\Gamma/\mathbb{Z}$ is torsion free. Applying Lemma \ref{homological stability} thus reduces the torsion free to the torsion case. 
\subsection{$\Gamma$ a ring, $\Gamma=\mathbb{Z}[\lambda,\lambda^{-1}]$ }
\label{subs gamma a ring}
Let $\lambda \in \mathbb{R}$ be algebraic with degree $d$. Consider the ring generated by $\lambda$, i.e. $\Gamma=\mathbb{Z}[\lambda,\lambda^{-1}]$. In this case, we have that $IE(\mathbb{Z}[\lambda,\lambda^{-1}])$ is canonically isomorphic to the measure-preserving subgroup of the irrational slope Thompson group $V_{\lambda}$ with slope $\langle \lambda \rangle$ and breakpoints $\mathbb{Z}[\lambda,\lambda^{-1}]$ as studied in [\cite{stein1992groups}, \cite{cleary2000regular}, \cite{irrationalslope22}]. \ \\
It is notable that via the embedding $\mathbb{Z} \oplus \lambda \mathbb{Z} \hookrightarrow \mathbb{Z}[\lambda,\lambda^{-1}]$ we obtain an embedding of $IE(\mathbb{Z} \oplus \lambda \mathbb{Z})$, a simple finitely generated amenable group, into $V_\lambda$ for any irrational $\lambda$, a behavior not exhibited in the rational slope Thompson groups. Let us describe a concrete generating set for $D(IE(\mathbb{Z}[\lambda,\lambda^{-1}]))$. \ \\
\begin{lemma}
    Let $\lambda$ be irrational and $\Gamma=\mathbb{Z}[\lambda,\lambda^{-1}]$. Then there exists a countable, generating set $\{1, \lambda_i\}_{i=1}^n$ for $\Gamma$ such that $4/5<\lambda_i<\lambda_{i+1}<1/2$ for all $i$. Then, let $\sigma_i, \hat{\sigma_i}$ be as in Theorem \ref{generating set theorem}. $IE(\mathbb{Z}[\lambda,\lambda^{-1}])$ is generated by $S=\{\sigma_i ,\hat{\sigma}_i\}_{i \in \mathbb{N}}$. \label{generating set ring}
\end{lemma}
\begin{proof}
    Let $\Gamma_n=\bigoplus_{i=-n}^n \lambda^n \mathbb{Z}$. For each $n$, $\Gamma_n \cong \mathbb{Z}^{2n+1}$ hence by Theorem \ref{generating set theorem} and Lemma \ref{prescribed generating set} we have that for all $n$ there exists $2n+1$ generators of $\Gamma_n$ where one of them is 1, and the others are a collection of rationally independent irrational numbers $\{\lambda_i\}_{i=1}^{2n}$ in the interval $(4/5,1/2)$. Then, by iteratively applying the argument in the proof of Lemma \ref{prescribed generating set}, we may assume, that for all $n$ there is a generating set of $\Gamma_n$ given by 1 and a collecton of rationally independent irrational numbers $\{\lambda_i\}_{i=1}^{2n}$ such that $4/5<\lambda_i<\lambda_{i+1}<1/2$ for all $i$. Since $\Gamma= \bigcup_{n \in \mathbb{N}}\Gamma_n$, the proof follows. By Theorem \ref{generating set theorem} we also have that 
    $D(IE(\Gamma_n))=\langle \sigma_i, \hat{\sigma_i}  \; i=1,...,n \rangle$. The result then follows from the observation that $D(IE(\Gamma))=\cup_{n=1}^\infty D(IE(\Gamma_n))$.
\end{proof}
\begin{rmk}
Note that in certain cases, (for a concrete example if $\lambda=\sqrt[n]{a}$ for $a \in \mathbb{Q}$), we have that $Rank_\mathbb{Q}(\Gamma)$ is finite. In this case, $\Gamma$ is polycyclic, with generating set $\{\lambda^i\}_{i=1}^{Rank_\mathbb{Q}(\Gamma)}$. In this case, we may apply Theorem \ref{generating set k} to find a finite generating set. 
\end{rmk}

Let us compute homological information for examples of the form $\mathbb{Z}[\lambda,\lambda^{-1}]$. Let the minimal polynomial of $\lambda$ be given by: 
$\lambda^d+\sum_{i=0}^{d-1}a_i \lambda^i=0$ be the minimal polynomial of $\lambda$. Consider the map $\phi_\lambda$ given by:
    $$\phi_\lambda: \mathbb{Z}^d \cong \bigoplus_{i=0}^{d-1} \lambda^i\mathbb{Z} \rightarrow \bigoplus_{i=0}^{d-1} \lambda^i\mathbb{Z} \quad t \mapsto \lambda t $$
 
    Moreover, $\Gamma$ is the direct limit of the sequence:
    $$ \mathbb{Z}^{d} \xrightarrow{\phi_\lambda}  \mathbb{Z}^{d} \xrightarrow{\phi_\lambda}  \mathbb{Z}^{d} \xrightarrow{\phi_\lambda} \hdots $$
It follows that $H_*(\Gamma)=\lim_{\phi_\lambda}H_*(\mathbb{Z}^d)$. Then, $H_*(\Gamma)=0$ whenever $*>d$, and $H_d(\Gamma)=0$. Therefore, $IE(\Gamma)$ is not rationally acyclic by \ref{rational homology}.   
\begin{example}[$d=2$]
 If $d=2$, one obtains that  $H_3(\mathbb{Z}[\lambda,\lambda^{-1}])=0$ through the above computation.  Hence, Lemma \ref{ah conj} reduces to a (split) short exact sequence of abelian groups:
 $$0 \rightarrow  \mathbb{Z}[\lambda,\lambda^{-1}] \otimes \mathbb{Z}_2 \xrightarrow{I} IE(\mathbb{Z}[\lambda,\lambda^{-1}])_{ab} \xrightarrow{j} H_2(\mathbb{Z}[\lambda,\lambda^{-1}])$$
$(\mathbb{Z}[\lambda,\lambda^{-1}] \otimes \mathbb{Z}_2) \oplus H_2(\mathbb{Z}[\lambda,\lambda^{-1}]) \cong IE(\mathbb{Z}[\lambda,\lambda^{-1}])_{ab}$ 

\end{example}

\end{document}